\documentclass{amsart}

\usepackage{amstext,amsmath,amsthm,mathrsfs,amssymb}

\usepackage{txfonts}
\usepackage{enumitem}
\usepackage[T1]{fontenc}

\usepackage{hyperref}

\newcommand{\longra}{\ensuremath{\longrightarrow}}
\newcommand{\ra}{\ensuremath{\Rightarrow}}

\newcommand{\lra}{\ensuremath{\Leftrightarrow}}
\newcommand{\Longra}{\ensuremath{\Longrightarrow}}

\DeclareMathOperator{\supp}{supp}

\newcommand{\C}{\varmathbb{C}}
\newcommand{\R}{\varmathbb{R}}
\newcommand{\N}{\varmathbb{N}}

\newcommand{\K}{\varmathbb{K}}

\newcommand{\E}{\varmathbb{E}}

\newcommand{\norm}[1]{\left|\left|#1\right|\right|}

\newcommand{\ip}[2]{\langle #1, #2 \rangle}
\newcommand{\ipp}[2]{\left( #1 \,|\, #2 \right)}

\theoremstyle{plain}
\newtheorem{thm}{Theorem}[section]
\newtheorem{lemma}[thm]{Lemma}
\newtheorem{prop}[thm]{Proposition}
\newtheorem{cor}[thm]{Corollary}

\newtheorem{fact}[thm]{Fact}

\theoremstyle{definition}
\newtheorem{definitie}[thm]{Definition}
\newtheorem{ex}[thm]{Example}

\theoremstyle{remark}
\newtheorem{rmk}[thm]{Remark}

\numberwithin{equation}{section}

\begin{document}

\title[Banach Space-Valued Extensions of Linear Operators on $L^{\infty}$]{Banach Space-Valued Extensions of Linear \\ Operators on $L^{\infty}$}

\author{Nick Lindemulder}
\address{Delft Institute of Applied Mathematics \\
Delft University of Technology \\
P.O. Box 5031 \\
2600 GA Delft \\
The Netherlands}
\email{N.Lindemulder@tudelft.nl}



\begin{abstract}
Let $E$ and $G$ be two Banach function spaces, let $T \in \mathcal{L}(E,Y)$, and let $\ip{X}{Y}$ be a Banach dual pair.
In this paper we give conditions for which there exists a (necessarily unique) bounded linear operator $T_{Y} \in \mathcal{L}(E(Y),G(Y))$
with the property that
\[
\ip{x}{T_{Y}e} = T\ip{x}{e}, \quad\quad\quad e \in E(Y), x \in X.
\]

The first main result states that, in case $\ip{X}{Y} = \ip{Y^{*}}{Y}$ with $Y$ a reflexive Banach space, for the existence of $T_{Y}$
it sufficient that $T$ is dominated by a positive operator. We furthermore show that for $Y$ within a wide class of Banach spaces (including the Banach lattices) the validity of this extension result for $E=\ell^{\infty}$ and $G = \K$ even characterizes the reflexivity of $Y$.

The second main result concerns the case that $T$ is an adjoint operator on $L^{\infty}(A)$: we assume that $E = L^{\infty}(A)$ for a semi-finite measure space $(A,\mathscr{A},\mu)$, that $\ip{F}{G}$ is a K\"othe dual pair, and that $T$ is $\sigma(L^{\infty}(A),L^{1}(A))$-to-$\sigma(G,F)$ continuous. In this situation we show that $T_{Y}$ also exists provided that $T$ is dominated by a positive operator.
As an application of this result we consider conditional expectation on Banach space-valued $L^{\infty}$-spaces.
\end{abstract}

%%% ----------------------------------------------------------------------
\maketitle
%%% ----------------------------------------------------------------------
%\tableofcontents
\section{Introduction}

Given two measure spaces $(A,\mathscr{A},\mu)$ and $(B,\mathscr{B},\nu)$, $p,q \in [1,\infty]$,
a bounded linear operator $T \in \mathcal{L}(L^{p}(A),L^{q}(B))$, and a Banach space $Y$,
one can ask the question whether $T$ has a $Y$-valued extension $T_{Y} \in \mathcal{L}(L^{p}(A;Y),L^{q}(B;Y))$ in the sense that there exists a (necessarily unique) bounded linear operator $T_{Y} \in \mathcal{L}(L^{p}(A;Y),L^{q}(B;Y))$ satisfying
\begin{equation}\label{intro:eq:extension_problem}
\ip{T_{Y}f}{y^{*}} = T\ip{f}{y^{*}}, \quad\quad\quad f \in L^{p}(A;Y), y^{*} \in Y^{*}.
\end{equation}
Note that $T_{Y}$ (if it exists) extends the tensor extension $T \otimes I_{Y}$ of $T$, which is the linear operator from the algebraic tensor product $L^{p}(A) \otimes Y$ to the algebraic tensor product $L^{q}(B) \otimes Y$
determined by the formula
\[
(T \otimes I_{Y})(f \otimes y) = Tf \otimes y, \quad\quad\quad f \in L^{p}(A), y \in Y.
\]

For $p \in [1,\infty[$ it holds that $L^{p}(A) \otimes Y$ is dense in $L^{p}(A;Y)$, so that $T_{Y}$ is just the unique extension of $T \otimes I_{Y}$ to a bounded linear operator from $L^{p}(A;Y)$ to $L^{q}(B;Y)$.
It is well known that, in this case, the extension $T_{Y}$ exists if $T$ is dominated by a positive operator (i.e. there exists a positive operator $S \in \mathcal{L}(L^{p}(A),L^{q}(B))$ such that $|Tf| \leq S|f|$ for all $f \in L^{p}(A)$)
or $Y$ is (isomorphic to) a Hilbert space; this can, for instance, be found in \cite[Subsection~4.5.c]{Grafakos} (also see \cite{Rubio}).
Another extension result says that, if $p = q \in [1,\infty[$, $A=B$, and $Y$ is isomorphic to a closed linear subspace of a quotient of a space $L^{p}(C)$, then the extension $T_{Y}$ exists for every $T \in \mathcal{L}(L^{p}(A))$; see \cite{herz}.
There also exist examples in which $T_{Y}$ does not exist. In fact, for some operators $T$ the existence of the $Y$-valued extension $T_{Y}$ characterizes $Y$ as being isomorphic to a Hilbert space or characterizes different geometric properties of the Banach space $Y$; for example, the fact that the Fourier-Plancherel transform $\mathscr{F}$ on $L^{2}(\R^{d})$ has a $Y$-valued extension $\mathscr{F}_{Y}$ on $L^{2}(\R^{d};Y)$ if and only if $Y$ is isomorphic to a Hilbert space is due to Kwapi\'en \cite{Kwapien}, and the characterization of the UMD Banach spaces as those Banach spaces for which the Hilbert transform (on $L^{p}(\R^{d})$) has an extension to a bounded linear operator on $L^{p}(\R^{d};Y)$ for some/all $p \in ]1,\infty[$ is due to Burkholder \cite{Burkholder} (sufficiency of UMD) and Bourgain \cite{Bourgain} (necessity of UMD) (see also the survey paper \cite{Burkholder_survey}). For Banach space-valued extension results for singular integral operators (in the UMD setting) we refer to \cite{Hytonen1} (and the references therein).

It seems that the extension problem \eqref{intro:eq:extension_problem} has not been considered in the literature for $p = \infty$.
In this paper we will obtain analogues for $p=\infty$ of the just mentioned results for $p<\infty$ about Banach space-valued extensions of operators dominated by a positive operator and Hilbert space-valued extensions of arbitrary bounded linear operators; we will in fact consider the extension problem in more general settings then discussed in this introduction. In the Banach space setting we will mainly consider the extension problem in two directions.

The first direction is concerned with $Y$-valued extensions $T$ for $Y$ a reflexive Banach space,
with as main result in this direction (Theorem \ref{thm:main_result2_extension_reflexive}) the existence of $T_{Y}$ plus a norm estimate in case that $T$ is dominated by a positive operator. Via a result of Zippin \cite{Zippin}, which says that every separable reflexive Banach space embeds into a reflexive Banach space with a Schauder basis, we can reduce the situation to the case that $Y$ is a reflexive Banach space with a Schauder basis. This basis can then be used to define $T_{Y}$. We show that for the special case $A=\N$, $B=\{0\}$, so that $L^{\infty}(A) = \ell^{\infty}$ and $L^{q}(B) = \K$ (the scalar field),
and $Y \in \{c_{0},\ell_{1}\}$, the extension $T_{Y}$ fails to exist when $T \in \mathcal{L}(\ell^{\infty},\K) = (\ell^{\infty})^{*}$
is a Banach limit (so in particular $T \geq 0$). As a consequence of a generalization of a classical result due to Lozanovski on the reflexivity on Banach lattices we find that, given a Banach limit $T \in \mathcal{L}(\ell^{\infty},\K)$, for $Y$ within a large class of Banach spaces (including the Banach lattices), $Y$ is reflexive if and only if the $Y$-valued extension $T_{Y} \in \mathcal{L}(\ell^{\infty}(Y),Y)$ of $T$ exists (Corollary~\ref{cor:char_ext_Banach_limit}).

In the second direction we consider arbitrary $Y$ under the additional assumption that $T$ is an adjoint operator. To be more precise,
suppose that $(A,\mathscr{A},\mu)$ and $(B,\mathscr{B},\nu)$ are both $\sigma$-finite and that $q \in ]1,\infty]$, so that we have canonical isometric isomorphisms $L^{\infty}(A) \cong (L^{1}(A))^{*}$ and $L^{q}(B) \cong (L^{q'}(B))^{*}$ (with $\frac{1}{q}+\frac{1}{q'} = 1$). Let $T = S^{*} \in \mathcal{L}(L^{p}(A),L^{q}(B))$ be the adjoint of $S \in \mathcal{L}(L^{q'}(B),L^{1}(A))$ and let $Y$ be an arbitrary Banach space. As the main result (Theorem \ref{thm:main_result}) in this direction we will show, in case that $T$ is dominated by a positive operator, the existence of both $T_{Y}$ and $S_{Y^{*}}$ together with norm estimates plus the adjoint relation
\[
\int_{B}\ip{T_{Y}f}{g}\,d\nu = \int_{A}\ip{f}{S_{Y^{*}}g}\,d\mu, \quad\quad\quad f \in L^{\infty}(A;Y), g \in L^{q'}(B;Y^{*}).
\]
The idea is to first obtain $S_{Y^{*}}$ by bounded extension of $S \otimes I_{Y^{*}}$ and then show that the Banach space adjoint $(S_{Y^{*}})^{*}$ of this extension restricts to an operator $L^{\infty}(A;Y) \subset (L^{1}(A;Y^{*}))^{*} \longra L^{q}(B;Y) \subset (L^{q'}(B;Y^{*}))$, which is the desired extension $T_{Y}$.
An example and motivation for this extension problem is the conditional expectation operator on Banach space-valued $L^{\infty}$-spaces.

The paper is organized as follows. In Section 2 we will first treat some necessary preliminaries. In Section 3 we present the results of this paper, with a formulation of the general extension problem and some basics, in the second subsection the extension problem for
reflexive $Y$, in the third subsection the extension problem of adjoint operators on $L^{\infty}$ for general Banach dual pairs, and in the fourth (and last) subsection the extension problems in the Hilbert space setting. Next, the proof of Theorem \ref{thm:main_result} is given is Section 4. Finally, as an application and motivation, we consider the conditional expectation operator on Banach-valued $L^{\infty}$-spaces in Section 5.

\paragraph{\textbf{Conventions and notations.}} Throughout this paper we fix a field $\K \in \{ \R, \C \}$ and assume that all spaces are over this field $\K$. For a normed space $X$ we denote by $B_{X}$ its closed unit ball. We furthermore write $\ell^{p} = \ell^{p}(\N)$, $1 \leq p \leq \infty$. For two Banach lattices $E$ and $F$ we denote by $M(E,F)$ the set of all linear operators from $E$ to $F$ which are dominated by a positive operator.

\section{Preliminaries}

\subsection{Banach Dual Pairs}

For the general theory of dual systems we refer to \cite{schaefer}.

A \emph{Banach duality (pairing)} between two Banach spaces $X$ and $Y$ is a bounded bilinear form $\ip{\,\cdot\,}{\,\cdot\,}:X \times Y \longra \K$ for which the induced linear maps $x \mapsto \ip{x}{\,\cdot\,},X \longra Y^{*}$ and $y \mapsto \ip{\,\cdot\,}{y}, Y \longra X^{*}$ are injections.
A \emph{Banach dual pair} is a triple $(X,Y,\ip{\,\cdot\,}{\,\cdot\,})$ consisting of two Banach spaces $X$ and $Y$ together with Banach duality
$\ip{\,\cdot\,}{\,\cdot\,}$ between them. We write $\ip{X}{Y} = (X,Y,\ip{\,\cdot\,}{\,\cdot\,})$. We call a Banach dual pair $\ip{X}{Y}$ \emph{norming} if
\[
\norm{x} = \sup_{y \in B_{Y}}\ip{x}{y} \quad \mbox{and} \quad \norm{y} = \sup_{x \in B_{X}}\ip{x}{y}.
\]
Note that in this case $X$ can be viewed as a closed subspace of $Y^{*}$ and $Y$ can be viewed as a closed subspace of $X^{*}$.

Let $\ip{X}{Y}$ be a Banach dual pair. Then the locally convex Hausdorff topology on $X$ generated by the family of seminorms $\{\,|\ip{\,\cdot\,}{y}|\,\}_{y \in Y}$ is called the \emph{weak topology} on $X$ generated by the pairing $\ip{X}{Y}$ and is denoted by $\sigma(X,Y)$.
The \emph{weak topology} on $Y$ generated by $\ip{X}{Y}$ is defined similarly and is denoted by $\sigma(Y,X)$. The topological dual of $(X,\sigma(X,Y))$ and $(Y,\sigma(Y,X))$ are $Y$ and $X$, respectively; that is, $(X,\sigma(X,Y))' = \{ \ip{\,\cdot\,}{y} \mid y \in Y \}$ and
$(Y,\sigma(Y,X))' = \{ \ip{x}{\,\cdot\,} \mid x \in X \}$. We shall always make the identifications $(X,\sigma(X,Y))' = Y$ and $(Y,\sigma(Y,X))' = X$.

A linear subspace $Z$ of $Y$ is $\sigma(Y,X)$ dense in $Y$ if and only if $Z$ separates the points of $X$, i.e., for every nonzero $x \in X$ there exists a $z \in Z$ with $\ip{x}{z} \neq 0$.

Recall that $\sigma(X,X^{*})$ is called the \emph{weak topology} on $X$ and that $\sigma(X^{*},X)$  is called the \emph{weak}$^{*}$ topology on $X^{*}$.

Suppose that we are given two Banach dual pairs $\ip{X_{1}}{Y_{1}}$ and $\ip{X_{2}}{Y_{2}}$ and a linear operator $S$ from $X_{1}$ to $X_{2}$. Viewing $Y_{i}$ as vector subspace of the algebraic dual $X_{i}^{\#}$ of $X_{i}$ ($i=1,2$), $S$ is continuous as an operator $S:(X_{1},\sigma(X_{1},Y_{1})) \longra (X_{2},\sigma(X_{2},Y_{2}))$ if and only if its algebraic adjoint $S^{\#}$ maps
$Y_{2}$ into $Y_{1}$; we say that $S$ is $\sigma(X_{1},Y_{1})$-to-$\sigma(X_{2},Y_{2})$ continuous. In this situation, the restriction
$S':Y_{2} \longra Y_{1}$ of $S^{\#}$ is called the \emph{adjoint} of $S$ with respect to the dualities $\ip{X_{1}}{Y_{1}}$ and $\ip{X_{2}}{Y_{2}}$, and it is a $\sigma(Y_{2},X_{2})$-to-$\sigma(Y_{1},X_{1})$ continuous linear operator whose adjoint is $S''=(S')'=S$.
The operators $S$ and $S'$ are automatically bounded operators as a consequence of the closed graph theorem.
Finally, note that if we a priori know $S$ to be bounded and view $Y_{i}$ as vector subspace of the norm dual $X_{i}^{*}$ of $X_{i}$ ($i=1,2$), then $S$ is $\sigma(X_{1},Y_{1})$-to-$\sigma(X_{2},Y_{2})$ continuous if and only if its Banach space adjoint $S^{*}$ maps $Y_{2}$ into $Y_{1}$.

\subsection{Duality and Schauder Bases for Banach Spaces}

In Subsection 3.1 we will use Schauder bases in order to define Banach space-valued extensions of linear operators; for the basis of the theory of Schauder bases we refer to \cite{Albiac_Kalton}. The following well known facts will be important for us in this direction.

\begin{fact}\label{fact:bases_reflexivity}\
\begin{itemize}
\item[(I)] Let $X$ be a Banach space. If $X$ has Schauder basis $\{b_{n}\}_{n \in \N}$, then $X$ is reflexive if and only if
$\{b_{n}\}_{n \in \N}$ is both boundedly-complete and shrinking.
\item[(II)] Every separable reflexive Banach space is isomorphic to a closed linear subspace of a reflexive Banach space with a Schauder basis.
\item[(III)] Let $X$ be a closed linear subspace of a Banach lattice $E$. If $X$ is complemented in $E$ or $E$ has an order continuous norm,
then the following statements are equivalent:
\begin{itemize}
\item[(a)] $X$ is reflexive.
\item[(b)] $X$ does not have linear subspaces isomorphic to $c_{0}$ or $\ell^{1}$.
\end{itemize}
\end{itemize}
\end{fact}

A reference for (I) is \cite[Theorem~3.2.13]{Albiac_Kalton}. (II) is due to Zippin \cite{Zippin}.
(III) is a generalization due to Tzafriri and Meyer-Nieberg of a result of Lozanovski about the reflexivity of Banach lattices;
see \cite{Wotjowicz} (and the references therein). For a version of this reflexivity result for finitely generated Banach $C(K)$-modules
we refer to \cite{Kitover}.

\subsection{Riesz Spaces and Banach Lattices}

For to the theory of Riesz spaces and Banach lattices we refer to the books \cite{positive},
\cite{Meyer-Nieberg}. Let us recall the following notation, definitions and facts.

Given a measure space $(A,\mathscr{A},\mu)$, we denote by $L^{0}(A)=L^{0}(A,\mathscr{A},\mu;\K)$ the $\K$-Riesz space
of all $\mu$-a.e. equivalence classes of $\K$-valued $\mathscr{A}$-measurable functions on $A$ with its natural lattice operations.

We say that a linear operator $T:E \longra F$ between two Banach lattices is \emph{dominated} by a positive operator $S \in \mathcal{L}(E,F)$
if it holds that $|Te| \leq S|e|$ for all $e \in E$; we also say that $S$ is a dominant for $T$ and we write $T \lesssim S$.
We denote by $\mathrm{mat}(T)$ the set of all dominants of $T$; then $\mathrm{mat}(T) \subset \mathcal{L}_{b}(E,F)^{+}$. If there is a least element in $\mathrm{mat}(T)$ with respect to the ordering of $\mathcal{L}_{b}(E,F)$ then it is called the \emph{least dominant} of $T$ and is denoted by $|T|$. We denote by $M(E,F)$ the space of all linear operators $T:E \longra F$ for which $\mathrm{mat}(T) \neq \emptyset$. Then $M(E,F) \subset \mathcal{L}(E,F)$. For $T \in M(E,F)$ we define
\[
\norm{T}_{M(E,F)} := \inf\{ \,\norm{S}\,:\, S \in \mathrm{maj}(T) \}.
\]
Then $\norm{T} \leq \norm{T}_{M(E,F)}$ for all $T \in M(X,Y)$ and $\norm{T}_{M(E,F)}= \norm{\,|T|\,}$ whenever $|T|$ exists;
in particular $\norm{T} = \norm{T}_{M(E,F)}$ when $T \geq 0$.

A linear operator $T:E \longra F$ between two Banach lattices is called \emph{regular} if it is a linear combination of positive operators. We denote by $\mathcal{L}_{r}(E,F)$ the space of all such operators. Then we have $\mathcal{L}_{r}(E,F) \subset M(E,F) \subset \mathcal{L}(E,F)$
and we write $\norm{T}_{r} := \norm{T}_{M(E,F)}$ for $T \in \mathcal{L}_{r}(E,F)$. If $F$ is Dedekind complete, then we have
$\mathcal{L}_{r}(E,F) = M(E,F)$.

A Banach lattice $E$ is called a \emph{KB-space} (\emph{Kantorovich-Banach space}) if every increasing norm bounded sequence of $E^{+}$ is norm convergent. It is not difficult to see that a Banach lattice $E$ is a KB-space if and only if every increasing norm bounded net of $E^{+}$ is norm convergent. Every reflexive Banach lattice is an example of a KB-space. An other example is the Lebesgue space $L^{1}(A)$.

A Banach lattice $E$ is said to have a \emph{Levi norm} if every increasing norm bounded net of $E^{+}$ has a supremum in $E$.
When this property only holds for sequences, then we say that $E$ has a \emph{sequentially Levi norm}.
KB-spaces are examples of Banach lattices having a Levi norm. An other example is $L^{\infty}(A)$ on a Maharam measure space
$(A,\mathscr{A},\mu)$; see the next subsection for the notion of Maharam measure space. Note that a Banach lattice with a Levi norm
must be Dedekind complete

A Banach lattice $E$ is said to have a Fatou norm if $\sup_{\alpha}\norm{x_{\alpha}} = \norm{x}$ whenever $\{x_{\alpha}\}_{\alpha} \subset E$ is an increasing net with supremum $x$.

\subsection{Measure Theory}

\paragraph{General measure theory.} For the content of this paragraph we refer to \cite{fremlin2}.

A measure space $(A,\mathscr{A},\mu)$ is called
\begin{itemize}
\item \emph{semi-finite} if for every $B \in \mathscr{A}$ with $\mu(B)>0$ there exists a $C \subset B, C \in \mathscr{A}$ with $0<\mu(C)<\infty$;
\item \emph{decomposable} (or \emph{strictly localizable})\footnote{Such measure spaces are also said to satisfy the \emph{direct sum property}.} if there exists a family $\{A_{i}\}_{i \in I}$ of pairwise disjoint sets in $\mathscr{A}$ such that $\mu(A_{i}) \in ]0,\infty[$ for all $i \in I$, and for each $B \in \mathscr{A}$ of finite measure there exists countable subset $I_{0} \subset I$ of indices and a $\mu$-null set $N \in \mathscr{A}$ such that $A = \bigcup_{i \in I_{0}}(B \cap A_{i}) \cup N$;
\item \emph{Maharam} (or \emph{localizable}) if it is semi-finite and if for every $\mathscr{E} \subset \mathscr{A}$ there is a $H \in \mathscr{A}$ such that (i) $E \setminus H$ is neglible for every $E \in \mathscr{E}$ and (ii)
if $G \in \mathscr{A}$ and $E \setminus G$ is neglible for every $E \in \mathscr{E}$, then $H \setminus G$ is neglible.
\end{itemize}
Regarding the relation between the different types of measure spaces, the following chain of implications holds true
\cite[Theorem~211L]{fremlin2}:
\begin{equation}
\sigma-\mathrm{finite} \:\Longra\: \mathrm{decomposable} \:\Longra\: \mathrm{Maharam} \:\Longra\: \mathrm{semi-finite}.
\end{equation}
A more elegant equivalent definition of Maharam measure space is via the measure algebra of $(A,\mathscr{A},\mu)$, which is obtained from $\mathscr{A}$ by identifying sets which are $\mu$-a.e. equal: a measure space is Maharam if and only if its measure algebra is Maharam, i.e. is a semi-finite measure algebra which is Dedekind complete as a Boolean algebra (see \cite{fremlin_top-riesz&mt} and \cite{fremlin}).

The canonical linear map $g \mapsto \Lambda_{g}, L^{\infty}(A) \longra (L^{1}(A))^{*}$ is an injection if and only if
$(A,\mathscr{A},\mu)$ is semi-finite, in which case it is an isometry, and this map is a bijection if and only if $(A,\mathscr{A},\mu)$ is Maharam, in which case it is an isometric isomorphism; see \cite[Theorem~243G]{fremlin2}. The sufficiency of Maharamness in the latter statement is in fact a special case of Fact \ref{thm:Kothe_dual}.  Another important characterization of the Maharam measure spaces among the semi-finite measure spaces is \cite[Theorem~241.G.(b)]{fremlin2}: a semi-finite measure space $(A,\mathscr{A},\mu)$ is Maharam if and only if $L^{0}(A)$ is Dedekind complete.

\paragraph{Banach space-valued measurability.}

Let $(A,\mathscr{A},\mu)$ be a measure space and let $X$ be a Banach space.

We denote by
\[
\mathrm{St}(A;X) := \left\{ \sum_{j=1}^{n}1_{A_{j}} \otimes x_{j} \,:\, A_{j} \in \mathscr{A}\:\mathrm{disjoint}\:, x_{j} \in X \right\}
\]
the vector space of $X$-valued step functions; here we use the usual notational convention to view, given a function $f:A \longra \K$, $f \otimes x$ as the function $a \mapsto f(a)x,\, A \longra X$.
A function $f:A \longra X$ is called strongly measurable if it is the pointwise limit of a sequence $(f_{k})_{k \in \N} \subset \mathrm{St}(A;X)$; it can be shown that the sequence $(f_{k})_{k}$ can be chosen such that $\norm{f_{k}}_{X} \leq \norm{f}_{X}$.
The well known Pettis measurability theorem says that a function $f:A \longra X$ is strongly measurable if and only if $f$ is separably valued and $\ip{f}{x^{*}}$ is measurable for all $x^{*}$ in some weak$^{*}$ dense subspace $Z$ of $X^{*}$; consequently, if $f:A \longra X$ is strongly measurable and takes its values in a closed linear subspace $Y$ of $X$, then $f$ is also strongly measurable as a function $A \longra Y$.
We denote by $L^{0}(A;X)$ the vector space of all $\mu$-a.e. equivalence classes of strongly measurable functions $f:A \longra X$.
We also view $L^{0}(A;X)$ as the vector space of all $\mu$-a.e. equivalence classes
of functions $g:A \longra X$ which are $\mu$-a.e. equivalent to a strongly measurable function on $f:A \longra X$.

\subsection{Banach Function Spaces}

For the theory of Banach function spaces we refer to \cite{zaanen},\cite{Meyer-Nieberg} ($\sigma$-finite measure spaces) and
\cite{fremlin_top-riesz&mt},\cite{fremlin} (general measure spaces and, in particular, Maharam measure spaces).

A \emph{Banach function space} on $(A,\mathscr{A},\mu)$ is an ideal $E$ of $L^{0}(A)$ which is equipped with a Banach lattice norm.
Note that each Banach function space is $\sigma$-Dedekind complete, being an ideal in the $\sigma$-Dedekind complete $L^{0}(A)$.
Examples of Banach function spaces are the $L^{p}$-spaces ($p\in [1,\infty]$), Orlicz spaces, Lorentz spaces, and Marcienkiewicz spaces.

A \emph{K\"othe dual pair (of Banach function spaces)} on $(A,\mathscr{A},\mu)$ is a Banach dual pair $\ip{E}{F}$ consisting of two Banach functions spaces $E$ and $F$ on $(A,\mathscr{A},\mu)$ with $E \cdot F \subset L^{1}(A)$ for which the pairing $\ip{\,\cdot\,}{\,\cdot\,}$
is given by
\[
\ip{e}{f} = \int_{A} fg\,d\mu, \quad\quad\quad e \in E, f \in F.
\]
Observe that the induced linear maps $e \mapsto \ip{e}{\,\cdot\,},E \longra F^{*}$ and $f \mapsto \ip{\,\cdot\,}{f}, F \longra E^{*}$ are lattice isomorphisms onto their images.
Examples of K\"othe dual pairs are $\ip{L^{1}(A)}{L^{\infty}(A)}$ for $(A,\mathscr{A},\mu)$ semi-finite or $\ip{L^{p}(A)}{L^{p'}(A)}$ for $p,p' \in ]1,\infty[, \frac{1}{p}+\frac{1}{p'}=1$; these two examples are even norming.

The  \emph{K\"othe dual} of a Banach function space $E$ on $(A,\mathscr{A},\mu)$ is the ideal $E^{\times}$ of $L^{0}(A)$ defined by
\[
E^{\times} := \{ f \in L^{0}(A) \,:\, fe \in L^{1}(A) \: \forall e \in E \},
\]
and is equipped with the seminorm
\[
\norm{f}_{E^{\times}} := \sup\left\{ \: \left| \int_{A}fe\,d\mu \right| \,:\, e \in E, \norm{e} \leq 1 \: \right\}.
\]

Suppose that $(A,\mathscr{A},\mu)$ is Maharam. Then every Banach function space $E$ on $(A,\mathscr{A},\mu)$ is Dedekind complete, being an ideal in the Dedekind complete $L^{0}(A)$, and has a well-defined
\emph{support} or \emph{carrier} $\supp(E)$ in $A$, which is the smallest set $\supp(E)$ (with respect to $\mu$-a.e. inclusion) such that
every $e \in E$ vanishes $\mu$-a.e. on $A \setminus \supp(E)$. It holds that $\supp(E)=A$ if and only if $E$ is order dense in $L^{0}(A)$ if and only if for every $B \in \mathscr{A}$ there exists a $C \in \mathscr{A}$ such that $C \subset A$, $\mu(C)>0$, and $1_{C} \in E$. In situation we have the following important duality result:

\begin{fact}\label{thm:Kothe_dual}
Suppose that $E$ is a Banach function space on the Maharam measure space $(A,\mathscr{A},\mu)$ having full carrier (i.e. $\supp(E) = A$). Then $E^{\times}$ is a Banach function space on $(A,\mathscr{A},\mu)$ with $\supp(E^{\times}) = A$ and $\ip{E}{E^{\times}}$ is a K\"othe dual pair on $(A,\mathscr{A},\mu)$ for which the image of $f \mapsto \ip{\,\cdot\,}{f}, E^{\times} \longra E^{*}$ is the band of order continuous functionals in $E^{*}$. In particular, $f \mapsto \ip{\,\cdot\,}{f}, E^{\times} \longra E^{*}$ is an isometric lattice isomorphism if and only if $E$ has an order continuous norm.
\end{fact}

Note that $E=L^{\infty}(A)$ does in general not have an order continuous norm, in which case the norm dual $(L^{\infty}(A))^{*}$ has functionals which are not order continuous, or equivalently, functionals which do not belong to the K\"othe dual $(L^{\infty}(A))^{\times} = L^{1}(A)$. In the special case of the counting measure space $(A,\mathscr{A},\mu) = (\N,\mathcal{P}(\N),\#)$, so that $E=\ell^{\infty}$, examples of linear functionals belonging to $(\ell^{\infty})^{*} \setminus \ell^{1}$ are the so-called Banach limits, whose existence can be established using Hahn-Banach (see \cite[Section III.7]{conway}).
\begin{definitie}\label{def:Banach_limit}
A bounded linear functional $\Lambda \in (\ell^{\infty})^{*}$ is called a \emph{Banach limit} if it has the following properties:
\begin{itemize}
\item[(a)] If $\{x_{n}\}_{n \in \N} \in \ell^{\infty}$ is a convergent sequence with limit $x$ (as $n \to \infty$), then $\Lambda(\{x_{n}\}_{n \in \N}) = x$.
\item[(b)] $\Lambda$ is positive.
\item[(c)] $\Lambda(\{x_{n}\}_{n \in \N}) = \Lambda(\{x_{n+1}\}_{n \in \N})$ for all $\{x_{n}\}_{n \in \N} \in \ell^{\infty}$.
\end{itemize}
\end{definitie}
We will use Banach limits as a counterexample to the extension problem in Subsection \ref{subsection:ext_reflexive}; see Example \ref{ex:counterexample}.

\subsection{K\"othe-Bochner Spaces}

Given a Banach function space $E$ on a measure space $(A,\mathscr{A},\mu)$, we define the vector space
\[
E(X) := \{ f \in L^{0}(A;X) \,:\, \norm{f}_{X} \in E \}.
\]
Endowed with the norm $\norm{f}:=\norm{\norm{f}_{X}}_{E}$, $E(X)$ becomes a Banach space which is called the \emph{K\"othe-Bochner space}
associated with $E$ and $X$.
We denote by $E \tilde{\otimes} X$ the closure of $E \otimes X$ in $E(X)$; recall that we use the usual convention to view $e \otimes x$ as the function $a \mapsto e(a)x$. We have $E(X) = E \tilde{\otimes} X$ provided that $E$ has an order continuous norm; in fact,
it is not diffcult to show that the linear subspace
\begin{align*}
\mathrm{St}_{E}(A;X)
&:= \left\{ \sum_{j=1}^{n}1_{A_{j}} \otimes x_{j} \,:\, A_{j} \in \mathscr{A}\:\mathrm{disjoint}\:, 1_{A_{j}} \in E,x_{j} \in X
\right\}  \\
&= \mathrm{St}(A;X) \cap E(X)
\end{align*}
of step functions which are in $E(X)$ is already dense in $E(X)$ provided that $E$ has an order continuous norm (see \cite{Hensgen}). We would like to mention that there are several cross-norms on $E \otimes X$ which coincide with the restricted norm coming from $E(X)$ (see \cite{Hensgen} and the references therein).

Observe that for $E=L^{p}(A)$ ($p \in [1,\infty]$) we get the usual Lebesgue-Bochner space $E(X) = L^{p}(A;X)$.

If $\ip{X}{Y}$ is a Banach dual pair and $\ip{E}{F}$ a K\"othe dual pair on $(A,\mathscr{A},\mu)$, then $\ip{E(X)}{F(Y)}$,
$\ip{E \tilde{\otimes} X}{F(Y)}$, $\ip{E(X)}{F \tilde{\otimes} Y}$ and $\ip{E\tilde{\otimes}X}{F\tilde{\otimes}Y}$ are Banach dual pairs under the pairing
\[
\ip{e}{f} = \ip{e}{f}_{\ip{E(X)}{F(X)}} := \int_{A}\ip{e(a)}{f(a)}_{\ip{X}{Y}}\,d\mu(a);
\]
in fact $E \otimes X$ and $F \otimes Y$ are already separating for $F(Y)$ and $E(X)$, respectively.

Suppose that $\ip{X}{Y}$ is norming.
If $\ip{E}{F} = \ip{L^{1}(A)}{L^{\infty}(A)}$ with $(A,\mathscr{A},\mu)$ semi-finite or $\ip{E}{F}=\ip{L^{p}(A)}{L^{p'}(A)}$
with $1<p,p'<\infty, \frac{1}{p}+\frac{1}{p'}=1$, then the Banach dual pair $\ip{E(X)}{F(Y)}$ is norming; note that for the latter it suffices to consider the $\sigma$-finite case.
In the case of a semi-finite measure space $(A,\mathscr{A},\mu)$ it can in fact be shown (with a slight modification of the proof of \cite[Theorem~1.1]{buhvalov2}) that, if $\ip{E}{F}$ is a norming K\"othe dual pair on $A$, then the dual pair $\ip{F(X)}{E(Y)}$ is norming as well.

\subsection{Banach Space-Valued Extensions of Linear Operators Between Banach Function Spaces}

Given two Banach function spaces $E$ and $G$, a bounded linear operator $S$ from $E$ to $G$ and a Banach space $X$, we can define the tensor extension $S \otimes I_{X}$ from $E \otimes X$ to $G \otimes X$ as the linear operator determined by the formula
\[
(S \otimes I_{X})(f \otimes x) := Sf \otimes x, \quad\quad\quad f \in E, x \in X.
\]
It is a natural question whether $S \otimes I_{X}$ extends to a bounded linear operator from $E \tilde{\otimes} X$ to $G \tilde{\otimes} X$; recall that $F \tilde{\otimes} X$ denotes the closure of $F \otimes X$ in $F(X)$ when $F$ is a Banach function space.
If $S \lesssim R$ for a positive operator $R \in \mathcal{L}(E,G)$ ($R \geq 0$ dominates $S$), then it can be shown that
\begin{equation}\label{eq:X-norm_estimate_tensor_extension_regular_operator}
\norm{(S \otimes I_{X})e}_{X} \leq R\norm{e}_{X}
\end{equation}
for all $f \in E \otimes X$ (cf. Lemma 2.3 of \cite{haase}), from which it is immediate that:

\begin{fact}\label{thm:ext_reg}
Let $S$ be a bounded linear operator between two Banach function spaces $E$ and $G$
and let $X$ be a Banach space.
If $S \in M(E,G)$ (i.e. $S$ is dominated by a positive operator), then $S \otimes I_{X}$ has a unique extension to a bounded linear operator $S_{X}$ from
$E \tilde{\otimes} X$ to $G \tilde{\otimes} X$ of norm $\norm{S_{X}} \leq \norm{S}_{M(E,G)}$.
\end{fact}

Note that if $E$ has an order continuous norm, so that $E \tilde{\otimes} X = E(X)$ (i.e. $E \otimes X$ is dense in $E(X)$), then the fact says that, for every $S \in M(E,G)$, the tensor extension $S \otimes I_{X}$ extends to a bounded linear operator $S_{X} \in \mathcal{L}(E(X),G(X))$, or equivalently, there exists a (necessarily unique) bounded linear operator $S_{X} \in \mathcal{L}(E(X),G(X))$
with the property that
\[
\ip{S_{X}e}{x^{*}} = S\ip{e}{x}, \quad\quad\quad e \in E(X), x^{*} \in X^{*}.
\]
The aim of this paper is to obtain analogues of this extension result (in the latter formulation) for $E$ not (necessarily) having an order continuous norm, with as main interest $E=L^{\infty}(A)$. Our two main results in this direction are
Theorem \ref{thm:main_result2_extension_reflexive} and Theorem \ref{thm:main_result}.

In case $G$ has a Levi norm (so that $G$ must be Dedekind complete and thus $M(E,G) = \mathcal{L}_{r}(E,G)$) the converse of the above fact holds as well and is an easy consequence of the fact taken from \cite{buhvalov1} that, in this case, $S$ is regular if and only if there exists a constant $C \geq 0$ such that, for all $e_{1},\ldots,e_{N} \in E$,
\[
\norm{\sum_{n=1}^{N}|Se_{n}|}_{G} \leq C\norm{\sum_{n=1}^{N}|e_{n}|}_{E}.
\]
\begin{fact}\label{thm:ext_reg_Levi}
Let $S$ be a bounded linear operator between two Banach function spaces $F$ and $G$ of which $G$ has a Levi norm.
Then the following assertions are equivalent.
\begin{itemize}
\item[(a)] $S$ is regular;
\item[(b)] $S \otimes I_{\ell^{1}}$ has an extension to (a necessarily unique) bounded linear operator $S_{\ell^{1}} \in \mathcal{L}(E \tilde{\otimes} \ell^{1}, G \tilde{\otimes} \ell^{1})$;
\item[(c)] $S \otimes I_{X}$ has an extension to (a necessarily unique) bounded linear operator $S_{X} \in \mathcal{L}(E \tilde{\otimes} X, G \tilde{\otimes} X)$ for every Banach space $X$.
\end{itemize}
In this situation we have $\norm{S_{X}} \leq \norm{S}_{r} \leq \norm{S_{\ell^{1}}}$.
\end{fact}

In case that $X = H$ is a Hilbert space, $E=L^{p_{1}}(A)$ and $G=L^{p_{2}}(B)$ with $1 \leq p_{1},p_{2} < \infty$, we do not need to
impose any restrictions on the operator $S$ for $S \otimes I_{H}$ to have a bounded extension.
This result was proved in the 1930's by Marcinkiewicz and Zygmund using Gaussian techniques \cite{Marcinkiewicz}: in fact, there exists a constant $0 < K \leq \max\{\frac{\norm{\gamma}_{p_{1}}}{\norm{\gamma}_{p_{2}}},1\}$, where $\gamma$ denotes a standard Gaussian random variable, such that, for all operators $S \in \mathcal{L}(L^{p_{1}}(A),L^{p_{2}}(B))$, $S \otimes I_{H}$ has a bounded extension $S_{H} \in \mathcal{L}(L^{p_{1}}(A;H),L^{p_{2}}(B;H))$ of norm $\norm{S_{H}} \leq K\norm{S}$ for any Hilbert space $H$, or equivalently,
we have the following square function estimate
\[
\norm{ \left(\sum_{k=1}^{n}|Se_{n}|^{2}\right)^{1/2} }_{L^{p_{2}}(B)} \leq K\norm{S}\norm{ \left(\sum_{k=1}^{n}|e_{n}|^{2}\right)^{1/2} }_{L^{p_{1}}(A)},
\]
valid for all $e_{1},\ldots,e_{n} \in L^{p_{1}}(A)$ (see also \cite{Rubio}). Using the Grothendieck inequality,
Krivine \cite{krivine} showed that this inequality is in fact valid for general Banach lattices with as best possible constant $K$ (working for all pairs of Banach lattices) the Grothendieck constant $K_{G}$ (also see \cite[pg.~82]{class_bs_II}). As a consequence:
\begin{fact}\label{thm:ext_Hilbert}
Let $S$ be a bounded linear operator between two Banach function spaces $E$ and $G$
and let $H$ be a Hilbert space. Then $S \otimes I_{H}$ has a bounded extension $S_{H} \in \mathcal{L}(E \tilde{\otimes} H,G \tilde{\otimes} H)$ of norm $\norm{S_{H}} \leq K_{G}\norm{S}$.
\end{fact}

Again note (as after Fact \ref{thm:ext_reg}) that if $E$ has an order continuous norm, then the result says that there exists a (necessarily unique) bounded linear operator $S_{H} \in \mathcal{L}(E(H),G(H))$ with the property that
\[
\ipp{S_{H}e}{h}_{H} = S\ipp{e}{h}_{H}, \quad\quad\quad e \in E(X), h \in H.
\]
We will extend this result to general $E$ not having an order continuous norm under a mild assumption on $G$ (Proposition \ref{prop:ext_H-valued}); moreover,
we will show that if $S$ is an adjoint operator, then so is $S_{H}$ (Corollary \ref{cor:ext_H-valued}).

\subsection{When are all Bounded Linear Operators Regular?}

Regarding Banach space-valued extensions of operators between Banach function spaces, in view of Fact \ref{thm:ext_reg} (and Fact \ref{thm:ext_reg_Levi}) it is interesting to know between which Banach function spaces every bounded linear operator is automatically regular.
For the following Banach lattice theoretic result in this direction we refer to \cite{positive_Handbook} and \cite{wickstead_regular_operators}.

\begin{fact}\label{thm:all_bdd_op_reg}
Let $E$ and $F$ be two Banach lattices. In each of the following cases we have that every bounded linear operator from $E$ to $F$ is regular:
\begin{itemize}
\item[(i)] $F$ is Dedekind complete and has a strong order unit.
\item[(ii)] $E$ is lattice isomorphic to an AL-space and $F$ has a Levi norm.
\item[(iii)] $E$ is lattice isomorphic to an atomic AL-space.
\item[(iv)] $E$ is atomic with order continuous norm and $F$ is an AM-space.
\end{itemize}
Moreover, in case (i) and (ii), if $F$ has a Fatou norm, then we have $\norm{T} = \norm{T}_{reg}$ for all $T \in \mathcal{L}(E,F)$.
\end{fact}

Note that for example every bounded linear operator $T:L^{p}(A) \longra L^{\infty}(B)$, $p \in [1,\infty[$ and $B$ Maharam, is regular by (i) and that every bounded linear operator $T:L^{1}(A) \longra L^{q}(B)$, $q \in [1,\infty[$, is regular by (ii), and thus have $Y$-valued extensions $T_{Y}$ of norm $\norm{T_{Y}} \leq \norm{T}$ for every Banach space $Y$ (by Fact \ref{thm:ext_reg}).

\section{Results}\label{sec:results}

\subsection{The Extension Problem}

Let $E$ and $G$ be two Banach function spaces and let $T \in \mathcal{L}(E,G)$. Given a Banach dual pair $\ip{X}{Y}$, we are interested in the question whether there exists a (necessarily unique) bounded linear operator $T_{Y} \in \mathcal{L}(E(Y),G(Y))$ with the property that
\begin{equation}\label{eq:the_extension_problem_relation}
\ip{x}{T_{Y}e} = T\ip{x}{e}, \quad\quad\quad e \in E(Y), x \in X.
\end{equation}
We call the operator $T_{Y}$ the $Y$-valued extension of $T$ with respect to the pairing $\ip{X}{Y}$.

In case $E$ has an order continuous norm, so that $E \otimes Y$ is dense in $E(Y)$ (i.e.\ $E \tilde{\otimes} Y = E(Y)$),
$T_{Y}$ is just the unique extension of $T \otimes I_{Y}$ to a bounded linear operator $T_{Y} \in \mathcal{L}(E(Y),G(Y))$.
So, in this situation, we have existence of $T_{Y}$ provided that $T$ is dominated by a positive operator (Fact \ref{thm:ext_reg}) or $Y$ is a Hilbert space (Fact \ref{thm:ext_Hilbert}). In this paper we will consider the extension problem \eqref{eq:the_extension_problem_relation} for $E$ not (necessarily) having an order continuous norm, with as main interest $E = L^{\infty}(A)$, and obtain analogues of the two just mentioned extension results; see Theorem \ref{thm:main_result2_extension_reflexive} and Theorem \ref{thm:main_result} for extensions of operators dominated by a positive operator
and Proposition \ref{prop:ext_H-valued} and Corollary \ref{cor:ext_H-valued} for Hilbert space-valued extensions.

\begin{rmk}\label{rmk:eq;the_extension_problem_relation_mapping}
Note that if $T_{Y}$ is a mapping $E(Y) \longra G(Y)$ satisfying \eqref{eq:the_extension_problem_relation}, then $T_{Y}$ is automatically a linear operator which is bounded by the closed graph theorem. Moreover, if $T \in M(E,G)$ and $\ip{X}{Y}$ is norming, then we have the norm estimate $\norm{T_{Y}} \leq \norm{T}_{M(E,G)}$.
\end{rmk}
\begin{proof}[Proof of the norm estimate]
Let $e \in E(Y)$ be given.
Pick a positive operator $R \in \mathcal{L}(E,G)$ dominating $T$. Since $\ip{X}{Y}$ is norming, we can pointwise estimate
\[
\norm{T_{Y}e}_{Y} = \sup_{x \in B_{X}}|\ip{x}{T_{Y}e}| \stackrel{\eqref{eq:the_extension_problem_relation}}{=} \sup_{x \in B_{X}}|T\ip{x}{e}| \leq \sup_{x \in B_{X}}R|\ip{x}{e}| \leq R\norm{e}_{Y},
\]
and thus $\norm{T_{Y}e}_{G(Y)} \leq \norm{R}\norm{e}_{E(Y)}$. Therefore, $\norm{T_{Y}} \leq \norm{T}_{M(E,G)}$.
\end{proof}

The following simple lemma gives, in two situations, a suggestion how to obtain the $Y$-valued extension of $T$:

\begin{lemma}\label{lemma:if_ext_exists_then}
Let $E$ and $G$ be two Banach function spaces, $T \in \mathcal{L}(E,G)$, and $\ip{X}{Y}$ a Banach dual pair. Assume that $T$ has a $Y$-valued extension $T_{Y}$ with respect to $\ip{X}{Y}$.
\begin{itemize}
\item[(i)] If $Y$ has a Schauder basis $\{b_{n}\}_{n \in \N}$ with biorthogonal functionals $\{b_{n}^{*}\}_{n \in \N} \subset X$,
then we must have
\begin{equation}\label{eq:formula_via_basis}
T_{Y}e = \sum_{n=0}^{\infty}T\ip{b_{n}^{*}}{e} \otimes b_{n}
\end{equation}
pointwise in $Y$ for every $e \in E(Y)$.
\item[(ii)] Suppose that $\ip{D}{E}$ and $\ip{F}{G}$ are K\"othe dual pairs and that $T$ is $\sigma(E,D)$-to-$\sigma(G,F)$ continuous with adjoint $S \in \mathcal{L}(F,D)$. If $S \otimes I_{X}$ has an extension to a bounded linear operator
$S_{X} \in \mathcal{L}(F \tilde{\otimes} X, D \tilde{\otimes} X)$, then $T_{Y}$ is $\sigma(E(Y),D \tilde{\otimes} X)$-to-$\sigma(G(Y),F \tilde{\otimes} X)$ continuous with adjoint $S_{X}$.
\end{itemize}
\end{lemma}
\begin{proof}
(i) is immediate from the definition of Schauder basis and \eqref{eq:the_extension_problem_relation}.
For (ii), let $e \in E(Y)$. For $f \in F$ and $x \in X$ we compute
\begin{align*}
\ip{T_{Y}e}{f \otimes x}_{\ip{G(Y)}{F \tilde{\otimes} X}}
&= \ip{\ip{T_{Y}e}{x}_{\ip{Y}{X}}}{f}_{\ip{G}{F}}
\stackrel{\eqref{eq:the_extension_problem_relation}}{=} \ip{T\ip{e}{x}_{\ip{Y}{X}}}{f}_{\ip{G}{F}} \\
&= \ip{\ip{e}{x}_{\ip{Y}{X}}}{Sf}_{\ip{E}{D}} = \ip{e}{Sf \otimes x}_{\ip{E(Y)}{D \tilde{\otimes} X}},
\end{align*}
so that, by linearity,
\[
\ip{T_{Y}e}{\phi}_{\ip{G(Y)}{F \tilde{\otimes} X}} = \ip{e}{(S\otimes I_{X})\phi}_{\ip{E(Y)}{D \tilde{\otimes} X}} = \ip{e}{S_{X}\phi}_{\ip{E(Y)}{D \tilde{\otimes} X}}
\]
for all $\phi \in F \otimes X$. By continuity and density this identity extends to all $\phi \in F \tilde{\otimes} X$, proving the desired result.
\end{proof}

In the setting of (i) in this lemma, if the basis $\{b_{n}\}_{n \in \N}$ is boundedly-complete and if $X$ is the closed linear span of $\{b_{n}^{*}\}_{n \in \N}$ in $X^{*}$, then we can use formula \eqref{eq:formula_via_basis} to define $T_{Y}$:

\begin{lemma}\label{lemma:def_extension_via_boundedly-complete_basis}
Let $E$ and $G$ be two Banach function spaces and let $T \in \mathcal{L}(E,G)$. Suppose that $Y$ is a Banach space having a boundedly-complete Schauder basis $\{b_{n}\}_{n \in \N}$ with biorthogonal functionals  $\{b_{n}^{*}\}_{n \in \N}$. Define $X$ as the closed linear span of $\{b_{n}^{*}\}_{n \in \N}$ in $Y^{*}$. If $T \in M(E,G)$, then it has a $Y$-valued extension $T_{Y} \in \mathcal{L}(E(Y),G(Y))$
with respect to $\ip{X}{Y}$ (in the sense of \eqref{eq:the_extension_problem_relation}) of norm $\norm{T_{Y}} \leq \norm{T}_{M(E,G)}$.
\end{lemma}
\begin{proof}
Let $R \in \mathcal{L}(E,G)$ be a positive operator dominating $T$. For all $e \in E(Y)$ we can estimate
\begin{align*}
\norm{\sum_{n=0}^{N}T\ip{b_{n}^{*}}{e} \otimes b_{n}}_{Y}
&= \norm{(T \otimes I_{Y})\left( \sum_{n=0}^{N}\ip{b_{n}^{*}}{e} \otimes b_{n} \right) }_{Y} \stackrel{\eqref{eq:X-norm_estimate_tensor_extension_regular_operator}}{\leq} R\norm{\sum_{n=0}^{N}\ip{b_{n}^{*}}{e} \otimes b_{n}}_{Y} \\
&\leq K R\norm{e}_{Y},
\end{align*}
where $K$ is the basis constant of $\{b_{n}\}_{n \in \N}$. Since the basis $\{b_{n}\}_{n \in \N}$ is boundedly complete,
we can define $T_{Y}e \in L^{0}(A;Y)$ as the pointwise limit $\lim_{N \to \infty}\sum_{n=0}^{N}T\ip{b_{n}^{*}}{e} \otimes b_{n}$ in $Y$
to obtain an element $T_{Y}e \in G(Y)$ satisfying $\norm{T_{Y}e}_{Y} \leq KR\norm{e}_{Y} \in G$. It then clearly holds that
$\ip{b_{n}^{*}}{T_{Y}e} = T\ip{b_{n}^{*}}{e}$ for all $e \in E(Y)$ and $n \in \N$, from which it follows that, in fact,
\[
\ip{x}{T_{Y}e} = T\ip{x}{e}, \quad\quad\quad e \in E(Y), x \in X.
\]
Remark \ref{rmk:eq;the_extension_problem_relation_mapping} now completes the proof.
\end{proof}

In the situation of the above lemma, the canonical map $j:Y \longra X^{*}$ given by $j(y)(x) = \ip{y}{x}$, for all $y \in Y$ and $x \in X$, is an isomorphism, which is isometric in case $\{b_{n}\}_{n \in \N}$ is monotone; see \cite[Theorem~3.2.10]{Albiac_Kalton}.
In particular, (possibly) up to an equivalence of norms,
the above lemma is concerned with a special case of the situation $\ip{X}{Y} = \ip{X}{X^{*}}$.
Regarding general $Y$-valued extensions with respect to $\ip{X}{Y}=\ip{X}{X^{*}}$, let us remark the following:

\begin{rmk}\label{rmk:lemma:def_extension_via_boundedly-complete_basis;predual}
Let $E$ and $G$ be two Banach function spaces and let $T \in \mathcal{L}(E,G)$.
Let $X$ be a Banach space and put $Y:=X^{*}$.
In this situation we would like to simply define the $Y$-valued extension $T_{Y}$ of $T$ with respect to $\ip{X}{Y}$ by \eqref{eq:the_extension_problem_relation}.
However, $\{ \ip{x}{Te} : x \in X \} \subset G$ is just a family of equivalence classes of measurable functions and it is not clear how to obtain an element $T_{Y}e \in G(Y,\sigma(Y,X))$.
In case $G$ is a Banach function space over $(B,\mathscr{B},\nu) = (B,\mathcal{P}(B),\#)$ this problem does not occur.
Moreover, if $B$ is countable or $Y$ is separable, then we obtain an element $T_{Y}e \in G(Y)$.
\end{rmk}

In view of Lemma \ref{lemma:if_ext_exists_then}.(ii) it is natural to consider the extension problem in the following lemma.

\begin{lemma}\label{eq:relation_ext_S_T}
Let $\ip{D}{E}$ and $\ip{F}{G}$ be two K\"othe dual pairs and
let $T \in \mathcal{L}(E,G)$ be a $\sigma(E,D)$-to-$\sigma(G,F)$ continuous linear operator with adjoint $S \in \mathcal{L}(F,D)$. For any dual pair of Banach spaces $\ip{X}{Y}$, the following are equivalent:
\begin{itemize}
\item[(a)] $T \otimes I_{Y}$ extends to a (necessarily unique) $\sigma(E(Y),D \tilde{\otimes} X)$-to-$\sigma(G(Y),F \tilde{\otimes} X)$ continuous linear operator $T_{Y} \in \mathcal{L}(E(Y),G(Y))$.
\item[(b)] $S \otimes I_{X}$ extends to a (necessarily unique) $\sigma(F \tilde{\otimes} X,G(Y))$-to-$\sigma(D \tilde{\otimes} X, E(Y))$ continuous linear operator $S_{X} \in \mathcal{L}(F \tilde{\otimes} X,D \tilde{\otimes} X)$.
\end{itemize}
In this situation, $S_{X}$ and $T_{Y}$ are adjoints of each other and $T_{Y}$ is the $Y$-valued extension of $T$
with respect to $\ip{X}{Y}$ (in the sense of \eqref{eq:the_extension_problem_relation}).
\end{lemma}
\begin{proof}
Note that the uniqueness in (a) and (b) follow from the $\sigma(E(Y),D \tilde{\otimes} X)$-density of $E \otimes Y$ in $E(Y)$ and the $\sigma(F \tilde{\otimes} X,G(Y))$-density of $F \otimes X$ in $F \tilde{\otimes} X$. The adjoint part in the last statement is contained in the proof of the implications "(a)$\ra$(b)" and "(b)$\ra$(a)". That $T_{Y}$ then is the $Y$-valued extension of $T$ with respect to $\ip{X}{Y}$ can be seen as follows: Given $e \in E(Y)$ and $x \in X$, we have
\begin{align*}
\ip{f}{\ip{x}{T_{Y}e}_{\ip{X}{Y}}}_{\ip{F}{G}}
&= \ip{f \otimes x}{T_{Y}e}_{\ip{F \tilde{\otimes}X}{G(Y)}} =
\ip{Sf \otimes x}{e}_{\ip{D \tilde{\otimes} X}{E(Y)}} \\
&= \ip{Sf}{\ip{x}{e}_{\ip{X}{Y}}}_{\ip{D}{E}} = \ip{f}{S\ip{x}{e}_{\ip{X}{Y}}}_{\ip{F}{G}}
\end{align*}
for every $f \in F$. As $F$ is separating for $G$, this shows $\ip{x}{T_{Y}e} = T\ip{x}{e}$.

"(a)$\ra$(b)":  Let $S_{X}:=(T_{Y})^{'} \in  \mathcal{L}(F \tilde{\otimes} X,D \tilde{\otimes} X)$ be the the adjoint of $T_{Y}$. A computation similar to the above one yields that $\ip{S_{X}f}{e}_{\ip{D \tilde{\otimes} X}{E(Y)}} = \ip{(S \otimes I_{X})f}{e}_{\ip{D \tilde{\otimes} X}{E(Y)}} $ for all $f \in F \otimes X$ and $e \in E \otimes Y$ while $E \otimes Y$ separates the points of $D \tilde{\otimes} X$, whence $S_{X}f = (S \otimes I_{X})f$ for all $f \in F \otimes X$. This gives (b).

"(b)$\ra$(a)": Completely analogous to the implication "(a)$\ra$(b)".
\end{proof}

In the next two subsections we will use Lemma \ref{lemma:def_extension_via_boundedly-complete_basis} and Lemma \ref{eq:relation_ext_S_T}
to obtain our two main extension results, Theorem \ref{thm:main_result2_extension_reflexive} and Theorem \ref{thm:main_result}.

\subsection{Extensions with respect to $\ip{X}{Y} = \ip{Y^{*}}{Y}$ with $Y$ Reflexive}\label{subsection:ext_reflexive}

\begin{thm}\label{thm:main_result2_extension_reflexive}
Let $E$ and $G$ be two Banach function spaces and let $T \in M(E,G)$. If $Y$ is a reflexive Banach space, then the $Y$-valued extension $T_{Y} \in \mathcal{L}(E(Y),G(Y))$ of $T$ with respect to $\ip{Y^{*}}{Y}$ (in the sense of \eqref{eq:the_extension_problem_relation}) exists and is of norm $\norm{T_{Y}} \leq \norm{T}_{M(E,G)}$.
\end{thm}

As an immediate consequence of this theorem and Fact \ref{thm:all_bdd_op_reg} we have:
\begin{cor}
Let $E$ and $G$ be two Banach function spaces such that one of the following four conditions is satisfied:
\begin{itemize}
\item[(i)] $G$ has a strong order unit;
\item[(ii)] $E$ is lattice isomorphic to an AL-space and $G$ has a Levi norm;
\item[(iii)] $E$ is lattice isomorphic to an atomic AL-space;
\item[(iv)] $E$ is atomic with order continuous norm and $G$ is an AM-space.
\end{itemize}
Then, for every $T \in \mathcal{L}(E,G)$ and every reflexive Banach space $Y$, the $Y$-valued extension $T_{Y} \in \mathcal{L}(E(Y),G(Y))$ of $T$ with respect to $\ip{Y^{*}}{Y}$ exists. Moreover, in case of (i) and (ii), if $G$ has a Fatou norm, then we have the norm estimate $\norm{T_{Y}} \leq \norm{T}$.
\end{cor}

In combination with Fact \ref{fact:bases_reflexivity}.(II), the next lemma allows us to reduce the proof of the theorem to
the case that $Y$ is a reflexive Banach space with a Schauder basis.

\begin{lemma}\label{lemma:thm;main_result2_extension_reflexive_Reduction}
Let $E$ and $G$ be two Banach function spaces, $T \in \mathcal{L}(E,G)$, and $Y$ a Banach space.
\begin{itemize}
\item[(i)] If $T$ has a $U$-valued extension $T_{U}$ with respect to $\ip{U^{*}}{U}$ for every separable closed linear subspace $U$ of $Y$,
then $T$ also has a $Y$-valued extension $T_{Y}$ with respect to $\ip{Y^{*}}{Y}$.
\item[(ii)] If $Y$ is a closed linear subspace of a Banach space $Z$ for which $T$ has a $Z$-valued extension $T_{Z}$ with respect to
$\ip{Z^{*}}{Z}$, then $T$ also has $Y$-valued extension $T_{Y}$ with respect to $\ip{Y^{*}}{Y}$.
\end{itemize}
\end{lemma}
\begin{proof}
(i) This follows easily from the fact that every $f \in E(Y)$ may be viewed as an element of $E(U) \subset E(Y)$ for some separable closed linear subspace $U$ of $Y$ in combination with Remark \ref{rmk:eq;the_extension_problem_relation_mapping}.

(ii) Viewing $E(Y)$ as closed linear subspace of $E(Z)$, $T_{Z}$ restricts to an operator on $E(Y)$ as consequence of the fact that $Y = {^{\perp}}(Y^{\perp})$. From Hahn-Banach it follows that $T_{Y}:= T_{Z}\big|_{E(Y)}$ is a $Y$-valued extension of $T$ with respect to $\ip{Y^{*}}{Y}$.
\end{proof}

We are now ready to give a clean proof of the theorem.
\begin{proof}[Proof of Theorem \ref{thm:main_result2_extension_reflexive}]
First, in view of (i) of the above lemma and the fact that a closed linear subspace of a reflexive Banach space is a reflexive Banach space on its own right, it suffices to consider the case that $Y$ is a separable reflexive Banach space. 
Next, in view of (ii) of the above lemma and Fact~\ref{fact:bases_reflexivity}.(II), it is in turn enough to treat the case that $Y$ is a reflexive Banach space having a Schauder basis. By Fact~\ref{fact:bases_reflexivity}.(I), this basis is both boundedly-complete and shrinking. The existence of $T_{Y}$ now follows from an application
of Lemma \ref{lemma:def_extension_via_boundedly-complete_basis}; for the norm estimate we refer to Remark \ref{rmk:eq;the_extension_problem_relation_mapping}.
\end{proof}

\begin{rmk}
The use of Fact \ref{fact:bases_reflexivity}.(II) in the above proof can be avoided in the special case that $G$ is a Banach function space over $(B,\mathscr{B},\nu) = (B,\mathcal{P}(B),\#)$; see Remark~\ref{rmk:lemma:def_extension_via_boundedly-complete_basis;predual}.
\end{rmk}

In the next example we show that for the non-reflexive Banach spaces $Y=c_{0}$ and $Y=\ell^{1}$ the statement of Theorem \ref{thm:main_result2_extension_reflexive} does not hold. Note that in both cases $Y$ has a Schauder basis (the standard basis),
with in case $Y=c_{0}$ a basis which is shrinking but not boundedly-complete and in case $Y=\ell^{1}$ a basis which is boundedly-complete but not shrinking; also see Fact \ref{fact:bases_reflexivity}.(I).

\begin{ex}\label{ex:counterexample}
Let $E=\ell^{\infty}$ and $G=\K$. Take $T \in \mathcal{L}(\ell^{\infty},\K) = (\ell^{\infty})^{*}$ to be a Banach limit (see Definition \ref{def:Banach_limit}). Then $T$ is a positive operator, but for $Y \in \{c_{0},\ell^{1}\}$ the $Y$-valued extension $T_{Y} \in \mathcal{L}(\ell^{\infty}(Y),Y)$ of $T$ with respect to $\ip{Y^{*}}{Y}$ does not exist.
\end{ex}
\begin{proof}
Let us first treat the case $Y=c_{0}$.
To the contrary we assume that $T_{c_{0}}$ does exist. By Lemma \ref{lemma:if_ext_exists_then}.(i) (equation \eqref{eq:formula_via_basis}) we must then have
\[
(Tf_{k})_{k \in \N} = T_{c_{0}}f \in c_{0}
\]
for all $f=(f_{k})_{k \in \N} \in \ell^{\infty}(c_{0})$; here $f_{k}$ is the $k$-th coordinate in $c_{0}$ of $f$ (with respect to the standard basis). But for
\[
f=(f_{k})_{k \in \N} \in \ell^{\infty}(c_{0}) \quad \mbox{given by}\quad  f_{k} := 1_{\{k,k+1,\ldots\}}
\]
we have $(Tf_{k})_{k \in \N} = \mathbf{1} \notin c_{0}$, a contradiction.

Next we treat the case $Y=\ell^{1}$. We again assume to the contrary that $T_{\ell^{1}}$ does exist. By Lemma \ref{lemma:if_ext_exists_then}.(i) (equation \eqref{eq:formula_via_basis}) we must then have
\[
(Tf_{k})_{k \in \N} = T_{\ell^{1}}f \in \ell^{1}
\]
for all $f=(f_{k})_{k \in \N} \in \ell^{\infty}(\ell^{1})$; here $f_{k}$ is the $k$-th coordinate in $\ell^{1}$ of $f$ (with respect to the standard basis). But for
\[
f=(f_{k})_{k \in \N} \in \ell^{\infty}(\ell^{1}) \quad \mbox{given by}\quad  f_{k} := 1_{\{k\}}
\]
and $\mathbf{1} \in \ell^{\infty} = (\ell^{1})^{*}$ this yields
\[
0 = \sum_{k=0}^{\infty}0 = \sum_{k=0}^{\infty}Tf_{k} = \ip{\mathbf{1}}{T_{\ell^{1}}f} = T\ip{\mathbf{1}}{f} = T\mathbf{1} = 1,
\]
a contradiction.
\end{proof}

\begin{rmk}
In \cite{Banach_limits} it is shown that Banach spaces $1$-complemented in their bidual admit vector-valued Banach limits, whereas $c_{0}$ does not. Since a $Y$-valued extension with respect to a norming dual pair $\ip{X}{Y}$ of a Banach limit is a vector-valued Banach limit
on $Y$, the latter also gives an explanation for the failure of the extension for $Y=c_{0}$ in the above example. The case $Y=\ell^{1}$ in this example shows that a vector-valued Banach limit on $\ell^{1}$ cannot be obtained as an $\ell^{1}$-valued extension with respect to $\ip{\ell^{\infty}}{\ell^{1}}$ of a Banach limit; note, however, that $\ell^{1}$-valued extensions with respect to $\ip{c_{0}}{\ell^{1}}$ of Banach limits exist by Lemma \ref{lemma:def_extension_via_boundedly-complete_basis} (also see Remark~\ref{rmk:lemma:def_extension_via_boundedly-complete_basis;predual}). Finally, observe that, by Theorem \ref{thm:main_result2_extension_reflexive}, every reflexive Banach space $Y$ admits vector-valued Banach limits which are $Y$-valued extensions with respect to $\ip{Y^{*}}{Y}$ of Banach limits.
\end{rmk}

Combining this example with Fact \ref{fact:bases_reflexivity}.(III),
Lemma \ref{lemma:thm;main_result2_extension_reflexive_Reduction}.(ii), and Theorem \ref{thm:main_result2_extension_reflexive},
we see that, for $Y$ in a wide class of Banach spaces (including the Banach lattices), Theorem \ref{thm:main_result2_extension_reflexive}
even characterizes the reflexivity of $Y$:

\begin{cor}\label{cor:char_ext_Banach_limit}
Let $Y$ be a closed linear subspace of a Banach lattice $E$ such that: $Y$ is complemented in $E$ or $E$ has an order continuous norm.
Given a Banach limit $T \in \mathcal{L}(\ell^{\infty},\K) = (\ell^{\infty})^{*}$, the following statements are equivalent.
\begin{itemize}
\item[(a)] $Y$ is reflexive;
\item[(b)] $T$ has a $Y$-valued extension $T_{Y} \in \mathcal{L}(\ell^{\infty}(Y),Y)$ with respect to $\ip{Y^{*}}{Y}$;
\item[(c)] $Y$ does not have linear subspaces isomorphic to $c_{0}$ or $\ell_{1}$.
\end{itemize}
\end{cor}

\subsection{Extensions of Adjoint Operators on $L^{\infty}$ with respect to Arbitrary Banach Dual Pairs $\ip{X}{Y}$}\label{subsec:arb_pairs}

\begin{thm}\label{thm:main_result}
Let $(A,\mathscr{A},\mu)$ be a semi-finite measure space, let $\ip{F}{G}$ be a K\"othe dual pair of Banach function spaces over a measure space $(B,\mathscr{B},\nu)$, and
let $T \in \mathcal{L}(L^{\infty}(A),G)$ be a $\sigma(L^{\infty}(A),L^{1}(A))$-to-$\sigma(G,F)$ continuous linear operator, say with adjoint $S \in \mathcal{L}(F,L^{1}(A))$.
If $T \in M(L^{\infty}(A),G)$, then we have, for any dual pair of Banach spaces $\ip{X}{Y}$, that $T \otimes I_{Y}$ has a unique extension to a $\sigma(L^{\infty}(A;Y),L^{1}(A;X))$-to-$\sigma(G(Y),F \tilde{\otimes} X)$ continuous linear operator $T_{Y} \in \mathcal{L}(L^{\infty}(A;Y),G(Y))$.
In this situation, $T_{Y}$ is the $Y$-valued extensions of $T$ with respect to $\ip{X}{Y}$ and the adjoint $S_{X} \in \mathcal{L}(F \tilde{\otimes} X,L^{1}(A;X))$ of $T_{Y}$ is the unique bounded extension of $S \otimes I_{X}$.
Moreover, $S \in M(F,L^{1}(A)) = \mathcal{L}_{r}(F,L^{1}(A))$ and these extensions are of norm $\norm{S_{X}} \leq \norm{S}_{r}$ and $\norm{T_{Y}} \leq \norm{T}_{M(L^{\infty}(A),G)}$.
\end{thm}

We will give the proof of this theorem in the next section. In Section 5 we will use this theorem to obtain the conditional expectation operator on Banach space-valued $L^{\infty}$-spaces.

\begin{rmk}
Note that for $T \otimes I_{Y}$ to have an extension to a $\sigma(L^{\infty}(A;Y),L^{1}(A;X))$-to-$\sigma(G(Y),F \tilde{\otimes} X)$ continuous linear operator $T_{Y} \in \mathcal{L}(L^{\infty}(A;Y),G(Y))$ it is necessary that $S$ is regular.
Indeed, from Lemma \ref{eq:relation_ext_S_T} it then follows that $S \otimes I_{X}$ extends to a bounded operator $S_{X} \in \mathcal{L}(F \tilde{\otimes} X,L^{1}(A;X))$ for any Banach space $X$ (just take $\ip{X}{Y}=\ip{X}{X^{*}}$ as dual pair of Banach spaces), which by Fact \ref{thm:ext_reg_Levi} just means that $S$ is regular.

We will in fact start the proof of this theorem by showing that $S$ is regular, then extend $S \otimes I_{X}$ to a bounded linear operator
$S_{X} \in \mathcal{L}(F \tilde{\otimes} X,L^{1}(A;X))$ and obtain $T_{Y}$ by restriction of the Banach space adjoint $(S_{X})^{*} \in \mathcal{L}((L^{1}(A;X))^{*},(F \tilde{\otimes} X)^{*})$.
\end{rmk}

Next, we consider situations in which the extension of $T \otimes I_{X}$ in Theorem \ref{thm:main_result} is for free.
The idea is to impose conditions on $\ip{F}{G}$ which guarantee that $T$ is automatically regular, either via $T$ being a bounded linear operator from $L^{\infty}(A)$ to $G$ or via $S$ and the following little lemma:
\begin{lemma}\label{lemma:T_reg_if_S_is}
In addition to the assumptions of Theorem \ref{thm:main_result}, suppose that the image of $i:g \mapsto \ip{\,\cdot\,}{g}, G \longra F^{*}$ is a band in $F^{*}$. Then $T$ is regular provided that $S$ is regular.
\end{lemma}
\begin{proof}
First note that $i(G)$ is a projection band in the Dedekind complete $F^{*}$. Let $P$ be the associated band projection.
Since $i$ is a lattice isomorphism onto its image, this projection $P$ induces a positive linear map $\pi:F^{*} \longra G$ such that
$\pi \circ i = I_{G}$. Now note that $\pi \circ S^{*}:(L^{1}(A))^{*} \longra G$ extends $T$ and is regular if $S$ is so.
\end{proof}

Note that $G$ must be Dedekind complete, being lattice isomorphic to a band in the Dedekind complete $F^{*}$.

Examples of a K\"othe dual pairs satisfying the hypotheses of this lemma are $\ip{F}{G} = \ip{L^{p}(B)}{L^{q}(B)}$ with $1<p,q<\infty$, $\frac{1}{p} + \frac{1}{q} = 1$ on an arbitrary measure space $(B,\mathcal{B},\nu)$ or $\ip{F}{G} = \ip{F}{F^{\times}}$ with $F$ a Banach function space on a Maharam measure space $(B,\mathscr{B},\nu)$ having full carrier (see Fact \ref{thm:Kothe_dual}).

\begin{cor}\label{cor:cor_main_result}
Let $(A,\mathscr{A},\mu)$ be a semi-finite measure space, let $\ip{F}{G}$ be a K\"othe dual pair of Banach function spaces over a measure space $(B,\mathscr{B},\nu)$, and
let $T: L^{\infty}(A) \longra G$ be a $\sigma(L^{\infty}(A),L^{1}(A))$-to-$\sigma(G,F)$ continuous linear operator.
In each of the following cases $T$ is automatically regular:
\begin{itemize}
\item[(i)] $G$ is Dedekind complete and has a strong order unit.
\item[(ii)] The image of $g \mapsto \ip{\,\cdot\,}{g}, G \longra F^{*}$ is a band in $F^{*}$ and $F$ is lattice isomorphic to an AL-space.
\end{itemize}
As a consequence, in each of these cases we have that, for any dual pair of Banach spaces $\ip{X}{Y}$, $T \otimes I_{Y}$ has a unique extension to a $\sigma(L^{\infty}(A;Y),L^{1}(A;X))$-to-$\sigma(G(Y),F \tilde{\otimes} X)$ continuous linear operator $T_{Y} \in \mathcal{L}(L^{\infty}(A;Y),G(Y))$, which is the $Y$-valued extension of $T$ with respect to $\ip{X}{Y}$.
In this situation, denoting by $S \in \mathcal{L}(F,L^{1}(A))$ the adjoint of $T$ with respect to the dualities
$\ip{L^{1}(A)}{L^{\infty}(A)}$ and $\ip{F}{G}$ and by $S_{X} \in \mathcal{L}(F \tilde{\otimes} X,L^{1}(A;X))$ the adjoint of $T_{X}$ with respect to the dualities $\ip{L^{1}(A)(X)}{L^{\infty}(A;Y)}$ and $\ip{F \tilde{\otimes} X}{G(Y)}$, $S \in \mathcal{L}_{r}(F,L^{1}(A))$, and $S_{X}$ is the unique bounded extension of $S \otimes I_{X}$.
Moreover, these extensions are of norm $\norm{S_{X}} \leq \norm{S}_{r}$ and $\norm{T_{Y}} \leq \norm{T}_{M(L^{\infty}(A),G)}$.
\end{cor}
\begin{proof}
Case (i) is an immediate consequence of Fact \ref{thm:all_bdd_op_reg}, whereas case (ii) follows from a combination Fact \ref{thm:all_bdd_op_reg} and the above lemma.
\end{proof}

Examples of K\"othe dual pairs $\ip{F}{G}$ satisfying the hypothesis of this result are $\ip{F}{G} = \ip{L^{1}(B)}{L^{\infty}(B)}$ on a Maharam measure space $(B,\mathscr{B},\nu)$ and $\ip{F}{G} = \ip{L^{p}(B)}{L^{p'}(B)}$, $p,p \in [1,\infty]$, $\frac{1}{p}+\frac{1}{p'}$,
on a finite measure space $(B,\mathscr{B},\nu)$.

Finally, we give two situations (involving some extra assumptions on $\ip{G}{F}$) in which $T$ being regular is not only a sufficient condition but a necessary condition as well. The idea is to impose conditions on $\ip{G}{F}$ which allow us to obtain that $T$ is regular, either via an application of Fact \ref{thm:ext_reg_Levi} to $T$ or via an application of this theorem to $S$ in combination
with Lemma \ref{lemma:T_reg_if_S_is}.

\begin{prop}\label{prop:nec_reg}
Suppose, in addition to the assumptions of Theorem \ref{thm:main_result}, that either
\begin{itemize}
\item[(i)] $G$ has a Levi norm, or
\item[(ii)] the image of
$g \mapsto \ip{\,\cdot\,}{g}, G \longra F^{*}$ is a band in $F^{*}$.
\end{itemize}
Then $T$ must be regular if, for some dual pair of Banach spaces $\ip{X}{Y}$ with $Y = \ell^{1}$ in case (i) and $X = \ell^{1}$ in case (ii), $T \otimes I_{Y}$ has an extension to a $\sigma(L^{\infty}(A;Y),L^{1}(A;X))$-to-$\sigma(G(Y),F \tilde{\otimes} X)$ continuous linear operator $T_{Y} \in \mathcal{L}(L^{\infty}(A;Y),G(Y))$.
\end{prop}
\begin{proof}
Let us first consider case (i). Note that it, in particular, $T \otimes I_{\ell^{1}}$ has a unique extension to a bounded linear operator from $L^{\infty}(A) \tilde{\otimes} \ell^{1}$ to $G \tilde{\otimes} \ell^{1}$. Fact \ref{thm:ext_reg_Levi} now yields that $T$ is regular.

Next we consider case (ii).
In view Lemma~\ref{lemma:T_reg_if_S_is} it suffices to prove that $S$ is regular. By Fact~\ref{thm:ext_reg_Levi}
and the fact that $L^{1}(A)$ has a Levi norm, for this it is in turn enough to show that $S \otimes I_{\ell^{1}}$ has an extension to a bounded linear operator $S_{\ell^{1}} \in \mathcal{L}(F \tilde{\otimes} \ell^{1},L^{1}(A;\ell^{1}))$. But, by Lemma \ref{eq:relation_ext_S_T}, the adjoint of the $\sigma(L^{\infty}(A;Y),L^{1}(A;\ell^{1}))$-to-$\sigma(G(Y),F\tilde{\otimes}\ell^{1})$ continuous linear operator $T_{Y} \in \mathcal{L}(L^{\infty}(A;Y),G(Y))$ is an extension of $S \otimes I_{\ell^{1}}$.
\end{proof}

\begin{cor}
Suppose, in addition to the assumptions of Theorem \ref{thm:main_result}, that the image of $g \mapsto \ip{\,\cdot\,}{g}, G \longra F^{*}$ is a band in $F^{*}$. Then the following assertions are equivalent.
\begin{itemize}
\item[(a)] $T$ is regular.
\item[(b)] $S$ is regular.
\item[(c)] $T \otimes I_{\ell^{\infty}}$ has an extension to a bounded linear operator from $L^{\infty}(A) \tilde{\otimes} \ell^{\infty}$ to $G \tilde{\otimes} \ell^{\infty}$.
\item[(d)] $S \otimes I_{\ell^{1}}$ has an extension to a bounded linear operator from $F \tilde{\otimes}\ell^{1}$ to $L^{1}(S;\ell^{1})$.
\item[(e)] For any dual pair of Banach spaces $\ip{X}{Y}$, $T \otimes I_{Y}$ has an extension to an operator $T_{Y} \in \mathcal{L}(L^{\infty}(A;Y),G(Y))$ which is $\sigma(L^{\infty}(A;Y),L^{1}(A;X))$-to-$\sigma(G(Y),F \tilde{\otimes} X)$ continuous.
\item[(f)] For any dual pair of Banach spaces $\ip{X}{Y}$, $S \otimes I_{X}$ has an extension to an operator $S_{X} \in \mathcal{L}(F \tilde{\otimes} X,L^{1}(A;X))$ which is $\sigma(F \tilde{\otimes} X,F^{\times}(Y))$-to-$\sigma(L^{1}(A;X),L^{\infty}(A;Y))$ continuous.
\end{itemize}
In this situation, for which to occur it suffices that $G$ has a strong order unit, $S_{X}$ and $T_{Y}$ are adjoints of each other
with respect to the dualities $\ip{F \tilde{\otimes} X}{G(Y)}$ and $\ip{L^{1}(A;X)}{L^{\infty}(A;Y)}$.
Moreover, these extensions are of norm $\norm{S_{X}} \leq \norm{S}_{r}$ and $\norm{T_{Y}} \leq \norm{T}_{r}$.
\end{cor}
\begin{proof}
Note that $G$ must be Dedekind complete, being lattice isomorphic to a band in the Dedekind complete $F^{*}$.

"(a)$\ra$(c)": See Fact \ref{thm:ext_reg}.

"(c)$\ra$(d)": Viewing $L^{1}(A;\ell^{1})$ as a closed linear subspace of $(L^{\infty} \tilde{\otimes} \ell^{\infty})^{*}$ via the isometric embedding \[
i:L^{1}(A;\ell^{1}) \longra (L^{\infty}(A) \tilde{\otimes} \ell^{\infty})^{*}, h \mapsto \ip{h}{\,\cdot\,}_{\ip{L^{1}(A;\ell^{1})}{L^{\infty}(A;\ell^{\infty})}}\big|_{L^{\infty}(A) \tilde{\otimes} \ell^{\infty}},
\]
it is enough that $S \otimes I_{\ell^{1}}$ has an extension to a bounded linear operator from $F \tilde{\otimes} \ell^{1}$ to
$(L^{\infty} \tilde{\otimes} \ell^{\infty})^{*}$. For this let
\[
j:F \tilde{\otimes} \ell^{1} \longra (G \tilde{\otimes} \ell^{\infty})^{*},f \mapsto \ip{f}{\,\cdot\,}_{\ip{F(\ell^{1})}{G(\ell^{\infty})}}\big|_{G \tilde{\otimes} \ell^{\infty}}
\]
be the natural continuous inclusion and let $U \in \mathcal{L}((G \tilde{\otimes} \ell^{\infty})^{*},(L^{\infty}(A) \tilde{\otimes} \ell^{\infty})^{*})$ be the Banach spaced adjoint of the bounded extension of $T \otimes I_{\ell^{\infty}}$.
Now observe that $U \circ j$ extends $S \otimes I_{\ell^{1}}$.

"(d)$\ra$(b)": This follows from Fact \ref{thm:ext_reg_Levi} and the fact that $L^{1}(A)$ has a Levi norm.

"(b)$\ra$(a)": This is precisely Lemma \ref{lemma:T_reg_if_S_is}.

"(a)$\lra$(e)": Combine Theorem \ref{thm:main_result} with the above proposition.

"(e)$\lra$(f)": See Lemma \ref{eq:relation_ext_S_T}.

The final assertion follows from Theorem \ref{thm:main_result} and Corollary \ref{cor:cor_main_result}.
\end{proof}

\subsection{Extensions with respect to $\ip{X}{Y} = \ip{H^{*}}{H}$ for a Hilbert space $H$}\label{subsec:hilbert_pairs}

Similar to Fact \ref{thm:ext_Hilbert}, for the existence of the extension in Theorem \ref{thm:main_result2_extension_reflexive} we do not need to impose any conditions on $T$ under the extra assumption that $G$ has a sequentially Levi norm.

\begin{prop}\label{prop:ext_H-valued}
Let $E$ and $G$ be two Banach function spaces, $T \in \mathcal{L}(E,G)$, and $H$ a Hilbert space.
Suppose that $G$ has a sequentially Levi norm. Then $T$ has a $H$-valued extension $T_{H} \in \mathcal{L}(E(H),G(H))$
with respect to $\ip{H^{*}}{H}$ (in the sense of \eqref{eq:the_extension_problem_relation}) which is of norm $\norm{T_{Y}} \leq K_{G}\norm{T}$, where $K_{G}$ is the Grothendieck constant.
\end{prop}
\begin{proof}
We may without loss of generality assume that $H$ is separable, see Lemma \ref{lemma:thm;main_result2_extension_reflexive_Reduction}.(i).
Now choose an orthonormal basis $\{h_{n}\}_{n
\in \N}$ of $H$. Given an $e \in E(H)$, it suffices to show that $\sum_{n \in \N}T\ip{h_{n}}{e} \otimes h_{n}$ converges pointwise a.e. in $H$ to an element of norm $\leq K_{G}\norm{T}\norm{e}_{E(H)}$.  But this follows from the hypothesis that $G$ has a sequentially Levi norm in combination with the estimate
\[
\norm{\left( \sum_{n=0}^{N}|T\ip{h_{n}}{e}|^{2} \right)^{1/2}}_{G}
\leq K_{G}\norm{T}\norm{\left( \sum_{n=0}^{N}|\ip{h_{n}}{e}|^{2} \right)^{1/2}}_{E} \leq K_{G}\norm{T}\norm{e}_{E(H)};
\]
here we use the Grothendieck inequality for Banach lattices (see \cite[pg.~82]{class_bs_II}).
\end{proof}

As an immediate consequence of this proposition, Fact \ref{thm:ext_Hilbert}, and Lemma \ref{lemma:if_ext_exists_then}.(ii), we have something similar for Theorem \ref{thm:main_result}:

\begin{cor}\label{cor:ext_H-valued}
Let $\ip{D}{E}$ and $\ip{F}{G}$ be two K\"othe dual pairs,
let $T \in \mathcal{L}(E,G)$ be a $\sigma(E,D)$-to-$\sigma(G,F)$ continuous linear operator  with adjoint $S \in \mathcal{L}(F,L^{1}(A))$,
and let $H$ be a Hilbert space. Then it holds that $T \otimes I_{H}$ has a unique extension to a $\sigma(E(H),D \tilde{\otimes}H^{*})$-to-$\sigma(G(H),F \tilde{\otimes} H^{*})$ continuous linear operator $T_{H} \in \mathcal{L}(E(Y),G(H))$.
In this situation, $T_{H}$ is the $H$-valued extensions of $T$ with respect to $\ip{H^{*}}{H}$ and the adjoint $S_{H} \in \mathcal{L}(F \tilde{\otimes} H^{*},F \tilde{\otimes} H^{*})$ of $T_{H}$ is the unique bounded extension of $S \otimes I_{H^{*}}$.
Moreover, these extensions are of norm $\norm{S_{H^{*}}} \leq K_{G}\norm{S}$ and $\norm{T_{H}} \leq K_{G}\norm{T}$.
\end{cor}

\section{Proof of Theorem \ref{thm:main_result}}\label{sec:proof}

Let the notations and assumptions be as in Theorem \ref{thm:main_result}. For the proof of this theorem we need three lemmas.
Before we can state the first lemma, we have to define the notion of countable step function:
a function $f:A \longra Y$ is called a \emph{countable step function} if it is measurable and only assumes countably many values.
Note that such a function is strongly measurable and can (in fact) be written as the poinwise limit
$f=\sum_{k=0}^{\infty}1_{A_{k}}y_{k}$ with $(A_{k})_{k \in \N}$ a mutually disjoint sequence in $\mathscr{A}$ and $(y_{n})_{n \in \N}$
a sequence in $Y$.

\begin{lemma}\label{lemma:ct_step_dense}
The subspace of countable step functions lying in $L^{\infty}(A;Y)$ is dense in $L^{\infty}(A;Y)$.
\end{lemma}
\begin{proof}
See the proof of Proposition 1.9 in \cite{vakhania}.
\end{proof}

\begin{lemma}\label{lemma:pt_wk_limit}
Let $\ip{D}{E}$ be a K\"othe dual pair of Banach function spaces on a measure space $(C,\mathscr{S},\rho)$, suppose that $(e_{k})_{k \in \N} \subset E$ is such that $\sum_{k=0}^{\infty}|e_{k}|$ is in $E$, and let $(y_{k})_{k \in \N}$ be a bounded sequence in $Y$. Then $\sum_{k=0}^{\infty}e_{k}(c)y_{k}$ converges in $Y$ for a.a. $c \in C$ and the resulting function $e:C \to Y$, defined by
$e(c):= \sum_{k=0}^{\infty}e_{k}(c)y_{k}$, belongs to $E(Y)$ and is of norm $\norm{e} \leq \norm{(y_{k})}_{\infty}\norm{\sum_{k=0}^{\infty}|e_{k}|}$. Moreover, we have $e = \sum_{k=0}^{\infty}e_{k} \otimes y_{k}$ with convergence in the $\sigma(E(Y),D(X))$-topology.
\end{lemma}
\begin{proof}
Observing that, for a.a.\ $c \in C$,
\[
\sum_{k=0}^{\infty}\norm{e_{k}(c)y_{k}} \leq \norm{(y_{k})}_{\infty} \sum_{k=0}^{\infty}|e_{k}(c)|,
\]
we find that, for a.a. $c \in C$ , $e(c) = \sum_{k=0}^{\infty}e_{k}(c)y_{k}$ converges in $Y$ and $\norm{e(c)} \leq \norm{(y_{k})}_{\infty} \sum_{k=0}^{\infty}|e_{k}(c)|$.
Therefore, $e \in E(Y)$ with $\norm{e} \leq \norm{(y_{k})}_{\infty}\norm{\sum_{k=0}^{\infty}|e_{k}|}$.

To prove the final assertion, fix an $d \in D(X)$.
The sequence $L^{1}(C)$-sequence $(c \mapsto \ip{d(c)}{\sum_{k=N+1}^{\infty}e_{k}(c)y_{k}})_{N \in \N}$ converges a.e.\ to $0$ as $N \to \infty$ and is dominated by a scalar multiple of $\norm{d}_{X}\norm{\sum_{k=0}^{\infty}|e_{k}|}\norm{(y_{k})}_{\infty} \in L^{1}(C)$, so that
\[
\ip{d}{e-\sum_{k=0}^{N}e_{k} \otimes y_{k}} = \int_{C}\ip{d(c)}{\sum_{k=N+1}^{\infty}e_{k}(c)y_{k}}\,d\rho(c)
\stackrel{N \to \infty}{\longra} 0. \qedhere
\]
\end{proof}

\begin{lemma}\label{lemma:L1_band}
Viewing $L^{1}(A)$ as closed Riesz subspace of $(L^{\infty}(A))^{*}$, $L^{1}(A)$ is a band in $(L^{\infty}(A))^{*}$.
\end{lemma}
\begin{proof}
Recalling that $\ip{L^{1}(A)}{L^{\infty}(A)}$ is a norming K\"othe dual pair, we may view $L^{1}(A)$ and $L^{\infty}(A)$ as closed Riesz subspaces of $(L^{\infty}(A))^{*}$ and $(L^{1}(A))^{*}$, respectively.
Accordingly, let $J:(L^{\infty}(A))^{*} \stackrel{\simeq}{\longra} (L^{1}(A))^{**}/(L^{\infty}(A))^{\perp}$ be the canonical isometric lattice isomorphism and let $\pi:(L^{1}(A))^{**} \longra (L^{1}(A))^{**}/(L^{\infty}(A))^{\perp}$ be the natural map.

To see that $L^{1}(A)$ is an ideal in $(L^{\infty}(A))^{*}$, let $\Lambda \in (L^{\infty}(A))^{*}$ and $f \in L^{1}(A)$ be such that $0 \leq \Lambda \leq f$ in $(L^{\infty}(A))^{*}$. Then $f$ viewed as a functional on $(L^{1}(A))^{*}$ is positive and its restriction to $L^{\infty}(A)$ dominates the positive $\Lambda \in (L^{\infty}(A))^{*}$. Hence, $\Lambda$ has an extension to a functional $\tilde{\Lambda}$ on $(L^{1}(A))^{*}$ satisfying $0 \leq \tilde{\Lambda} \leq f$ in $(L^{1}(A))^{**}$. Since $L^{1}(A)$, having an order continuous norm, is an ideal in $(L^{1}(A))^{**}$, it follows that $\tilde{\Lambda} \in L^{1}(A)$. Therefore,
$\Lambda = J^{-1}(\pi(\tilde{\Lambda})) \in J^{-1}(\pi(L^{1}(A))) = L^{1}(A)$.

It remains to be shown that the ideal $L^{1}(A)$ in $(L^{\infty}(A))^{*}$ is also order closed in $(L^{\infty}(A))^{*}$. To this end,
let $\{f_{\alpha}\}_{\alpha} \subset L^{1}(A)$ be such that $0 \leq f_{\alpha} \nearrow \Lambda \in (L^{\infty}(A))^{*}$. Then we in particular have that $\{f_{\alpha}\}$ is an increasing positive norm bounded net in $L^{1}(A)$. From the fact that $L^{1}(A)$ has a Levi norm it now follows that $f_{\alpha} \nearrow f$ for some $f \in L^{1}(A)$. But then we must have $\Lambda = f \in L^{1}(A)$, as desired.
\end{proof}

We are now ready to prove Theorem \ref{thm:main_result}
\begin{proof}[Proof of Theorem \ref{thm:main_result}.]
We only need to establish existence of $T_{Y}$ and the norm estimates.

First observe that $S$ is regular. Indeed, letting $j:F \hookrightarrow G^{*}$ be the natural inclusion and letting
$\pi:(L^{\infty}(A))^{*} \longra L^{1}(A)$ be the map induced by Lemma \ref{lemma:L1_band}, we have that $S = \pi \circ T^{*} \circ j:F \longra L^{1}(A)$. Therefore, $S \in M(F,L^{1}(A)) = \mathcal{L}_{r}(F,L^{1}(A))$ as $\pi \geq 0$, $T^{*} \in M(G^{*},(L^{\infty}(A))^{*})$, and $j \geq 0$.

By Fact \ref{thm:ext_reg}, as $S \in \mathcal{L}(F,L^{1}(A))$ is regular,
$S \otimes I_{X}$ has an extension to an operator $S_{X} \in \mathcal{L}(F \tilde{\otimes} X,L^{1}(A;X))$ of norm $\norm{S}_{X} \leq \norm{S}_{reg}$. Letting $i:L^{\infty}(Y;X) \hookrightarrow L^{1}(A;X))^{*}$ and $j:G(Y) \hookrightarrow (F \tilde{\otimes} X)^{*}$ be the natural continuous inclusions,
we claim that (i) $(S_{X})^{*} \circ i$ extends $j \circ (T \otimes I_{Y})$ and (ii) $(S_{X})^{*}$ maps $i(L^{\infty}(Y;X))$ into $j(G(Y))$ and (iii) $j^{-1} \circ (S_{X})^{*} \circ i \in \mathcal{L}(L^{\infty}(A;Y),G(Y))$ of norm $\leq \norm{T}_{M(L^{\infty}(A),G)}$. Then (ii) tells us that $S_{X}$ has an adjoint $(S_{X})^{'}$ w.r.t. the dualities $\ip{F \tilde{\otimes} X}{G(Y)}$ and $\ip{L^{1}(A;X)}{L^{\infty}(A;Y)}$, which by (i) extends $T \otimes I_{Y}$. The norm inequality $\norm{T_{Y}} \leq \norm{T}_{M(L^{\infty}(A),G)}$ then follows from (iii).

For (i), let $h \in L^{\infty}(A) \otimes Y$ be arbitrary. Then an elementary computation shows that
$\ip{(T \otimes I_{Y})h}{f} = \ip{h}{S_{X}f}$ for all $f$ in the dense subspace $F \otimes X$ of $F \tilde{\otimes} X$, which by continuity extends to all $f \in F(X)$. This gives (i).

For (ii) and (iii) we denote by $V$ the linear space consisting of all countable step functions in $L^{\infty}(A;Y)$ equipped with the restricted norm of $L^{\infty}(A;Y)$. Let $R \in \mathcal{L}(L^{\infty}(A),G)$ be a positive operator dominating $T$ and
fix an arbitrary $h \in V$, say $h = \sum_{k=0}^{\infty}1_{A_{k}}y_{k}$ with $(A_{k})_{k \in \N}$ a mutually disjoint sequence in $\mathscr{A}$ and $(y_{k})_{k \in \N}$ a bounded sequence in $Y$.
Then note that, by Lemma \ref{lemma:pt_wk_limit}, $h = \sum_{k=0}^{\infty}1_{A_{k}} \otimes y_{k}$ with convergence in the $\sigma(L^{\infty}(A;Y),L^{1}(A;X))$-topology. From the weak$^{*}$ continuity of $(S_{X})^{*}$ as a Banach space adjoint operator and (i) it follows that
\begin{equation}\label{eq:adj_ext_S}
(S_{X})^{*}i(h) = \sum_{k=0}^{\infty}(S_{X})^{*}i(1_{A_{k}} \otimes y_{k}) \stackrel{(i)}{=}  \sum_{k=0}^{\infty}j(T1_{A_{k}} \otimes y_{k})
\end{equation}
with convergence in the weak$^{*}$-topology.

For the sequence $(T1_{A_{k}})_{k \in \N} \subset G$, $\sum_{k=0}^{\infty}|T1_{A_{k}}| \in G$ follows from the ideal property of $G$ and the estimate
\[
\sum_{k=0}^{\infty}|T1_{A_{k}}| \leq \sum_{k=0}^{\infty}R1_{A_{k}} = \lim_{K \to \infty}\sum_{k=0}^{K}R1_{A_{k}} =
\lim_{K \to \infty}R\left(\sum_{k=0}^{K}1_{A_{k}}\right) \leq R1 \in G.
\]
Via Lemma \ref{lemma:pt_wk_limit} we obtain convergence of the series $\sum_{k=0}^{\infty}T1_{A_{k}} \otimes y_{k}$ in $G(Y)$ w.r.t.
the $\sigma(G(Y),F \tilde{\otimes} X)$-topology together with a norm estimate of the resulting element of $G(Y)$:
\[
\norm{\sum_{k=0}^{\infty}T1_{A_{k}} \otimes y_{k}} \leq \norm{R1}\norm{(y_{k})}_{\infty} \leq \norm{R}\norm{h}.
\]
In combination with (\ref{eq:adj_ext_S}) this gives $(S_{X})^{*}i(h) \in j(G(Y))$ and $\norm{j^{-1}((S_{X})^{*}ih)} \leq \norm{R}\norm{h}$. As $h$ and $R$ were arbitrary, this shows that $(S_{X})^{*} \circ i$ maps $V$ into $G(Y)$ and that
$j^{-1} \circ ((S_{X})^{*} \circ i\big|_{V}) \in \mathcal{L}(V,G(Y))$ is of norm $\leq \norm{T}_{M(L^{\infty}(A),G)}$.
$V$ being a dense subspace of $L^{\infty}(A;Y)$ (see Lemma \ref{lemma:ct_step_dense}), $j^{-1} \circ ((S_{X})^{*} \circ i\big|_{V})$ has a unique extension to a bounded linear operator $Q$ from $L^{\infty}(A;Y)$ to $G(Y)$ of norm $\leq \norm{T}_{M(L^{\infty}(A),G)}$.
The observation that $j \circ Q$ and $(S_{X})^{*} \circ i$ coincide on the dense subspace $V$ of $L^{\infty}(A;Y)$ and consequently that
$j \circ Q = (S_{X})^{*} \circ i$ now yields (ii) and (iii).
\end{proof}

\section{An Application: Conditional Expectation on Banach Space-Valued $L^{\infty}$-Spaces}

Let $(A,\mathscr{A},\mu)$ be a measure space, $\mathscr{F} \subset \mathscr{A}$ a sub-$\sigma$-algebra, and $X$ a Banach space.
The conditional expectation operator on $L^{1}(A;X)$ with respect to $\mathscr{F}$ is the operator $\E^{1}_{\mathscr{F},X} \in \mathcal{L}(L^{1}(A))$ which assigns to an $f \in L^{1}(A;X)$ the unique $\E^{1}_{\mathscr{F},X}f \in L^{1}(A,\mathscr{F};X)$ satisfying
\begin{equation}
\int_{F} f\,d\mu = \int_{F} \E^{1}_{\mathscr{F},X}f\,d\mu, \quad\quad\quad F \in \mathscr{F};
\end{equation}
here we write $L^{1}(A,\mathscr{F};X)$ for the closed linear subspace of $L^{1}(A;X)$ consisting of all equivalence classes which have a strongly $\mathscr{F}$-measurable representative.
This operator is a contractive projection with range $L^{1}(A,\mathscr{F};X)$ and it can be obtained via bounded tensor extension of the conditional expectation operator $\E^{1}_{\mathscr{F}}$ on $L^{1}(A)$, which is a positive operator. We refer to \cite{haase}, where also pointwise convexity (Jensen-type) inequalities are proved for $X$-valued extensions of positive operators.

Now suppose that $(A,\mathscr{A},\mu)$ is semi-finite and that the restricted measure space $(A,\mathscr{F},\mu|_{\mathscr{F}})$ is Maharam; it can in fact be shown that $(A,\mathscr{A},\mu)$ is automatically semi-finite when $(A,\mathscr{F},\mu|_{\mathscr{F}})$ is Maharam. Given a Banach space $Y$, we would like to define the conditional expectation operator
on $L^{\infty}(A;Y)$ with respect to $\mathscr{F}$ as the operator $\E^{\infty}_{\mathscr{F},Y} \in \mathcal{L}(L^{\infty}(A;Y))$ which assigns to an $f \in L^{\infty}(A;Y)$ the unique $\E^{\infty}_{\mathscr{F},Y}f \in L^{\infty}(A,\mathscr{F};Y)$ satisfying
\begin{equation}\label{eq:def_cond_exp_L_infty}
\int_{F}f\,d\mu = \int_{F}\E^{\infty}_{\mathscr{F},Y}f\,d\mu, \quad\quad\quad F \in \mathscr{F}(\mu);
\end{equation}
here $\mathscr{F}(\mu) = \{ F \in \mathscr{F} \,:\, \mu(F) < \infty \}$.
In scalar case $Y=\K$ we can define $\E^{\infty}_{\mathscr{F}} = \E^{\infty}_{\mathscr{F},\K}$ by restriction to $L^{\infty}(A)$ of the Banach space adjoint $(\E^{1}_{\mathscr{F}})^{*} \in \mathcal{L}((L^{1}(A))^{*})$:
\begin{lemma}
The conditional expectation operator $\E^{1}_{\mathscr{F}}$ on $L^{1}(A)$ is a $\sigma(L^{1}(A),L^{\infty}(A))$-to-$\sigma(L^{1}(A),L^{\infty}(A))$ continuous linear operator whose adjoint is a positive contractive projection on $L^{\infty}(A)$ with range $L^{\infty}(A,\mathscr{F})$ satisfying the above definition of conditional expectation operator on $L^{\infty}(A)$.
\end{lemma}
\begin{proof}
Recall that $\E_{\mathscr{G}}^{1}$ is a positive contractive projection on $L^{1}(A)$ with range $L^{1}(A,\mathscr{G})$; so $L^{1}(A) = L^{1}(A,\mathscr{G}) \oplus U$ where $U:= (1-\E_{\mathscr{G}}^{1})L^{1}(A)$. Since $L^{\infty}(A,\mathscr{G}) = (L^{1}(A,\mathscr{G}))^{*}$ (as $(A,\mathscr{G},\mu|_{\mathscr{G}})$ is Maharam), it follows that $(\E_{\mathscr{G}}^{1})^{*}$ is a positive contractive projection on $(L^{1}(A))^{*} = L^{\infty}(A,\mathscr{G}) \oplus U^{*}$ with range $L^{\infty}(A,\mathscr{G})$. Identifying $L^{\infty}(A)$ with a closed subspace of $(L^{1}(A))^{*}$ ($(A,\mathscr{A},\mu)$ is semi-finite), $(\E_{\mathscr{G}}^{1})^{*}$ restricts to a contractive projection on $L^{\infty}(A)$ with range $L^{\infty}(A,\mathscr{G})$.
\end{proof}

We can now obtain $\E_{\mathscr{G}}^{\infty,Y}$ from $\E^{\infty}_{\mathscr{F}}$
via an application of Theorem \ref{thm:main_result}:
\begin{prop}
Suppose that $(A,\mathscr{A},\mu)$ is semi-finite and that the restricted measure space $(A,\mathscr{F},\mu|_{\mathscr{F}})$ is Maharam.
For every Banach space $Y$ we have existence of the conditional expectation operator $\E^{\infty}_{\mathscr{F},Y}$ on $L^{\infty}(A;Y)$
(see \eqref{eq:def_cond_exp_L_infty}). $\E^{\infty}_{\mathscr{F},Y}$ is a contractive projection on $L^{\infty}(A;Y)$
with range $L^{\infty}(A,\mathscr{F};Y)$. Moreover, if $\ip{X}{Y}$ is a Banach dual pair, then
we have
\begin{equation*}\label{eq:conditional_exp_duality}
\int \ip{f}{\E_{\mathscr{F},Y}^{\infty}g}\,d\mu = \int \ip{\E_{\mathscr{F},X}^{1}f}{g}\,d\mu, \quad\quad\quad f \in L^{1}(A;X), g \in L^{\infty}(A;Y).
\end{equation*}
\end{prop}
\begin{proof}
Since there always exists a Banach space $X$ for which there is Banach dual pairing $\ip{X}{Y}$ (just take $X=Y^{*}$),
we may prove the first and second assertion at the same time. Applying Theorem \ref{thm:main_result} to
$S=\E_{\mathscr{F}}^{1} \in \mathcal{L}(L^{1}(A))$ and $T=\E^{\infty}_{\mathscr{F}} \in \mathcal{L}(L^{\infty}(A))$,
we get contractions $S_{X} \in \mathcal{L}(L^{1}(A;X))$ and $T_{Y} \in \mathcal{L}(L^{\infty}(A;Y))$ with respect to the duality $\ip{L^{1}(A;X)}{L^{\infty}(A;Y)}$ such that $S_{X}$ is the unique bounded extension of $S \otimes I_{X}$ and
\begin{equation*}\label{eq:Banach-valued_ext_cond_exp_L_infty}
\ip{x}{T_{Y}f} = T\ip{x}{f}, \quad\quad\quad x \in X, f \in L^{\infty}(A;Y).
\end{equation*}
Recalling that $\E^{1}_{\mathscr{F},X} = S_{X}$, it is not difficult to see that we can take $\E^{\infty}_{\mathscr{F},Y} := T_{Y}$.
\end{proof}

For a different and more direct way to define the conditional expectation operator on Banach-valued $L^{\infty}$-spaces
we refer to \cite{dinc} (also see the references therein).

\subsection*{Acknowledgment}
The author would like to thank Mark Veraar for making him aware of Fact~\ref{fact:bases_reflexivity}.(II).

\bibliographystyle{plain}

\end{document}